\documentclass[
 reprint,
 amsmath,amssymb,
 aps,
]{revtex4-2}

\usepackage{graphicx}
\usepackage{dcolumn}
\usepackage{bm}

\usepackage{scalerel}
\usepackage{amsthm}
\usepackage{tikz}
\usetikzlibrary{svg.path}

\definecolor{orcidlogocol}{HTML}{A6CE39}
\tikzset{
  orcidlogo/.pic={
    \fill[orcidlogocol] svg{M256,128c0,70.7-57.3,128-128,128C57.3,256,0,198.7,0,128C0,57.3,57.3,0,128,0C198.7,0,256,57.3,256,128z};
    \fill[white] svg{M86.3,186.2H70.9V79.1h15.4v48.4V186.2z}
                 svg{M108.9,79.1h41.6c39.6,0,57,28.3,57,53.6c0,27.5-21.5,53.6-56.8,53.6h-41.8V79.1z M124.3,172.4h24.5c34.9,0,42.9-26.5,42.9-39.7c0-21.5-13.7-39.7-43.7-39.7h-23.7V172.4z}
                 svg{M88.7,56.8c0,5.5-4.5,10.1-10.1,10.1c-5.6,0-10.1-4.6-10.1-10.1c0-5.6,4.5-10.1,10.1-10.1C84.2,46.7,88.7,51.3,88.7,56.8z};
  }
}

\newcommand\orcidicon[1]{\href{https://orcid.org/#1}{\mbox{\scalerel*{
\begin{tikzpicture}[yscale=-1,transform shape]
\pic{orcidlogo};
\end{tikzpicture}
}{|}}}}

\usepackage{bm} 
\newcommand{\bx}{{\bf x}}

\newcommand{\bX}{{\bf X}}

\newcommand{\bA}{{\bf A}}

\newcommand{\bB}{{\bf B}}
\newcommand{\bC}{{\bf C}}
\newcommand{\bCxx}{{\bC_{\bx\bx}}}
\newcommand{\bI}{{\bf I}}
\newcommand{\bJ}{{\bf J}}
\newcommand{\Jf}{{\bJ_f}}
\newcommand{\bK}{{\bf K}}
\newcommand{\Kobs}{{K_\text{obs}}}
\newcommand{\bL}{{\bf L}}

\newcommand{\bQ}{{\bf Q}}
\newcommand{\bW}{{\bf W}}
\newcommand{\bmxi}{{\bm{\xi}}}
\newcommand{\bmeta}{{\bm{\eta}}}
\newcommand{\bmtau}{{\bm{\tau}}}

\newcommand{\Alim}{{A_{\text{LIM}}}}

\newcommand{\Qlim}{{Q_{\text{LIM}}}}

\newcommand{\Acglim}{{A_{\text{C-LIM}}}}
\newcommand{\Bcglim}{{B_{\text{C-LIM}}}}
\newcommand{\Qcglim}{{Q_{\text{C-LIM}}}}

\newcommand{\Admd}{{A_{\text{DMD}}}}
\newcommand{\Ldmd}{{L_{\text{DMD}}}}
\newcommand{\taucglim}{{\tau_{\text{C-LIM}}}}
\newcommand{\LFP}{\textbf{L}_{\textit{FP}}}
\newcommand{\pst}{P_{\textit{st}}}

\usepackage{algorithm,algpseudocode}
\algnewcommand{\To}{\textbf{To }}
\algnewcommand\Input{\item[\textbf{Input:}]}%
\algnewcommand\Output{\item[\textbf{Output:}]}%

\usepackage{appendix}

\usepackage{lipsum}
\usepackage{amsfonts}
\usepackage{graphicx}
\usepackage{epstopdf}
\ifpdf
  \DeclareGraphicsExtensions{.eps,.pdf,.png,.jpg}
\else
  \DeclareGraphicsExtensions{.eps}
\fi

\newcommand{\R}{\mathbb{R}}

\newtheorem{theorem}{Theorem}
\newtheorem{corollary}{Corollary}[theorem]

\usepackage{amsmath} 
\usepackage{mathtools} 
\usepackage{makecell}
\usepackage{multirow}
\usepackage{booktabs}
\usepackage{afterpage}
\usepackage{graphicx} 
\usepackage{array}

\usepackage{tikz}
\usetikzlibrary{arrows.meta}
\tikzset{mynode/.style={draw, very thick, circle, minimum size=0.4cm}, myarrow/.style={very thick}}

\usepackage{graphicx}

\DeclareMathOperator*{\argmin}{arg\,min} 

\begin{document}

\preprint{APS/123-QED}

\title{Colored-LIM: A Data-Driven Method for Studying \\ Dynamical Systems with Temporally Correlated Stochasticity}

\author{Justin Lien}
 \email{(Corresponding author) lien.justin.t8@dc.tohoku.ac.jp}
\affiliation{%
 Mathematical Institute, Tohoku University
}%

\author{Yan-Ning Kuo}
\affiliation{
 Department of Earth and Atmospheric Sciences, Cornell University 
}%

\author{Hiroyasu Ando}
\affiliation{%
 Advanced Institute for Materials Research, Tohoku University 
\\ }%
\affiliation{%
 Mathematical Institute, Tohoku University
}%

\author{Shoichiro Kido}
\affiliation{
 Japan Agency for Marine-Earth Science and Technology 
}%


\date{\today}

\begin{abstract}
In real-world problems, environmental noise is often idealized as Gaussian white noise, despite potential temporal dependencies. 
The Linear Inverse Model (LIM) is a class of data-driven methods that extract dynamic and stochastic information from finite time-series data of complex systems. 
In this study, we introduce a new variant of LIM, called Colored-LIM, which models stochasticity using Ornstein-Uhlenbeck colored noise. 
Despite the non-trivial correlation between observable and colored noise, we show that Colored-LIM unveils the desired information merely from the correlation function of the observable. 
Therefore, this approach not only accounts for the memory effects of environmental noise, traditionally represented by time-uncorrelated white noise in the Classical LIM framework, but does so using the same observation dataset without requiring additional data. 
Furthermore, we show that Colored-LIM does not reduce to Classical LIM in the white noise limit, underscoring the importance of temporal dependencies in stochastic systems.

In this paper, we rigorously develop the Colored-LIM, explore its connections with the Classical LIM and Dynamic Mode Decomposition, and validate its effectiveness on both ideal linear and nonlinear systems. 
In addition, we illustrate the potential applications and implications of Colored-LIM for real-world problems, including the El Ni\~{n}o-Southern Oscillation and the electricity network of Tohoku University.
\end{abstract}

\maketitle


\section{Introduction} \label{Chap:Intro}

The stochastic differential equation (SDE) is a mathematical tool that incorporates the deterministic dynamics and random forcing which models the environmental noise of the systems. 
Widely applied to thermal physics, financial mathematics, and climate sciences in which the governing systems are influenced by uncertainty or random fluctuations, the development of SDEs has led to a significant influence on our understanding of complex systems \cite{Oksendal1987,Penland1989,Seifert2012,sarkka_solin_2019}. 
The Gaussian white noise often serves as an idealized representation of the random forcing due to its statistical properties and mathematical simplicity (e.g. uncorrelated in time and independent of frequency).
One of the simplest SDEs is an Ornstein-Uhlenbeck (OU) process, a linear mean-reverting process in which the deterministic dynamics and random forcing do not directly interact.
It has been extensively studied by both mathematicians and applied scientists, giving rise to several models and algorithms including the classical linear inverse model (LIM).

The Classical LIM is an empirical method widely used in climate sciences and geophysics that extracts the system dynamics from the observation by approximating the underlying complex system with an OU process (or a linear Markov model) \cite{Kwasniok2022,Penland1996,Penland1993}.
More precisely, in the Classical LIM framework, the data is collected from a system of the form 
\begin{align}\label{Eq:GeneralEq}
    \frac{d}{dt}\bx = F(\bx,t,\mu, \textit{noise}),
\end{align}
where $\bx(t)$ is the system state at time $t$, $F$ describes the stochastic dynamical system, and $\mu$ denotes the parameters.
The system $F$ is in general mathematical intractable, so an equation-free perspective is taken by constructing a white-noise driven linear system (see Eq.~(\ref{Eq:WienerProcess})) 
that estimates Eq.~(\ref{Eq:GeneralEq}), allowing the study of the time evolution of the departure of the equilibrium point. 
For this linearized system, the correlation function is an exponential function whose exponent is the linear dynamics, which can be extracted simply by the covariance and a lag-$\rho$ autocovariance for any $\rho > 0$. 
Then the diffusive behavior simply follows from the fluctuation-dissipation relation (FDR) \cite{Penland1989}.
However, in the application, the estimated dynamics may strongly depend on the choice of the lag $\rho$, indicating the potential deviation of approximating the system as a linear Markov model due to non-linearity or the non-white noise \cite{Penland1989}.
This observation has driven the development of variants of LIMs and other new approaches to the inverse problem \cite{Martinez2017,Shin2021}.

Since the assumptions of the Classical LIM might not work in practice, given the potential existence of non-linearity and non-Gaussianity in the system, many approaches have been proposed to model the behavior of complex stochastic dynamical systems. 
Examples include the Kramers-Moyal coefficient method, the Mori-Zwanzig formalism, and the multilevel regression modeling \cite{Barrie1986,Cui2018,Hassanibesheli2020,KONDRASHOV2015,Kravtsov2005,mccullagh2019,Niemann2008,Schmitt2006}. However, each of these methods has its limitations.
For example, the Kramers-Moyal coefficient method makes the white-noise assumption and therefore, performs poorly for non-Markovian systems \cite{Cui2018,Hassanibesheli2020}. 
The Mori-Zwanzig method projects the complex system to a lower dimensional space, obtaining exact equations for the observables only by putting the unresolved variables into a temporal kernel and residuals \cite{KONDRASHOV2015,Niemann2008}. 
The multilevel regression method as a generalization of multiple linear regression and the multiple polynomial regression, iteratively models the residuals at the current level by the variables at this level and all preceding levels; additional levels are introduced until the residual at the last level satisfies a white-noise condition \cite{Kravtsov2005}.
These methods effectively reproduce the probability density function of the underlying stochastic processes whenever applicable, and can even reconstruct the system parameters under the white-noise condition. However, for a non-Gaussian noise-driven stochastic process, they are limited to reproducing the probability density function.
Therefore, our objective is to propose a novel method capable of estimating the linear dynamics under non-Gaussian noise conditions.

Among the class of non-Gaussian noises, the normalized colored noise $\bmeta$ is one of the simplest time-correlated noises, exhibiting an exponential correlation function of the form 
\begin{equation}\label{Eq:Defining-Colored-Noise}
    \langle \bmeta(t+s) \bmeta(t)^T \rangle = \frac{1}{2\bmtau} e^{ \frac{-\left\vert s \right\vert}{\bmtau} }\bI
\end{equation} 
where $\bI$ denotes the identity matrix, for some constant $\bmtau>0$ called the noise correlation time \cite{KlosekDygas1988}. 
The colored noise $\bmeta$ models the environmental stochasticity with temporal dependency whose $e$-folding time is represented by $\bmtau$.
However, Eq.~(\ref{Eq:Defining-Colored-Noise}) along with the zero mean condition does not uniquely define the colored noise, as the higher-order moment is left unconstrained \cite{Gradziuk2022}. 
Several approaches to realizing the colored noise have been proposed, and they may have significantly different properties \cite{Benedetti2014,FALSONE2018,Gradziuk2022,Zhivomirov2018}.
As we focus on the continuous-time stochastic processes, the most natural choice is the OU colored noise, which can be realized as the steady state of the following SDE
\[ \frac{d}{dt} \bmeta = \frac{-1}{\bmtau}\bmeta + \frac{1}{\bmtau}\bmxi, \]
where $\bmxi$ is the $n$-dimensional normalized Gaussian random vector.

In this article, we propose a new variant of LIMs for stochastic processes with temporally correlated random forcing, called the Colored-LIM, which approximates the non-linear complex system (\ref{Eq:GeneralEq})
with an OU colored-noise driven linear system (\ref{Eq:CG-StochasticProcess}). 
For the linearized system, the correlation function of the observables is no longer an exponential function due to the non-trivial correlation between the state variable $\bx$ and the colored-noise random forcing $\bmeta$.
Nevertheless, for a given noise correlation time $\bmtau$, we utilize the correlation function of $\bx$ only and its derivatives up to the third order to formulate the inverse problem as solving a system of linear equations.
Also, we propose a minimization algorithm to determine $\bmtau$ uniquely whenever $\bmtau$ is not predetermined. 
In practice, the Colored-LIM algorithm amounts to computing the correlation function near the origin, applying numerical differentiation such as a finite difference scheme to compute the derivatives, and then solving a system of linear equations and a potential $1$-dimensional minimization problem, leading to its computational efficiency and numerical stability. 

Dynamic mode decomposition (DMD) is a deterministic counterpart to LIMs, providing a method to estimate the linear dynamics of unknown systems and extract their dominant dynamic modes. 
Over time, DMD has been extended into several variants and has found applications across a wide range of sciences and engineering \cite{Kutz2016,Schmid2022}.
Notably, one of the fundamental DMD variants, projected DMD, has been shown to be equivalent to Classical LIM in terms of its dynamical interpretation \cite{Tu2014}.
Therefore, the Colored-LIM can be viewed as an extension of DMD that operates in the presence of colored noise, offering a powerful tool for dynamical mode discovery.

The structure of the article is as follows:
In sections \ref{Chap:LIM}, \ref{Chap:Colored-LIM-Math}, and \ref{Chap:Colored-LIM-algorithm}, we focus on the linear stochastic processes, review the mathematical background of the Classical LIM, and develop the Colored-LIM by using the Fokker-Planck equation and operator. 
In sections \ref{Chap:Application-viewpoint} and \ref{Chap:Conn-with-DMD}, we discuss the Colored-LIM from an application perspective and explore its connection with DMD. 
In sections \ref{Chap:NumExp} and \ref{Chap:Real-world}, we demonstrate the effectiveness and applications of Colored-LIM to both ideal and real-world problems. 

Before moving to the next section, we briefly explain the convention of our notation. The vector is assumed to be a column vector without explicitly stated. For the theoretical development, quantities such as the system parameters, random variables, and their statistics, are denoted in bold.
When referring to time-series observation data, we mean a discrete-time sequence of vectors with an equal sampling interval represented by $\Delta t$.
We use the regular font for observation data and the model output.

\section{A Review of the Classical LIM} \label{Chap:LIM}

In this section, we review the mathematical framework of Classical LIM \cite{Penland1989,Penland1996,Penland1994}. The Classical LIM consists of first finding the dynamical matrix and then the diffusion matrix through the FDR. 
To find the dynamical matrix, several methods have been proposed \cite{Kravtsov2005,Penland1994}. 
We focus on the method of transition matrix, which reveals the fact that linear dynamics is characterized by the correlation function. 
In addition, we provide a proof of FDR, though elementary, by using the adjoing Fokker-Planck operator. 
These will inspire the development of Colored-LIM in the next section.

\subsection{The method of transition function}\label{Chap:TransFunc}

In the context of SDEs, the transition matrix (sometimes called the propagator, the Green's function) is an important concept in the analysis of stochastic processes, especially in the study of time-evolution of probability density \cite{sarkka_solin_2019}. 
It provides a systematic method for studying linear equations by describing the probability distribution of the solution of the SDE at a specific time, given an initial condition. 
In the Classical LIM, it characterizes the correlation function, and provides the ground for estimating the dynamics from data \cite{Penland1994}.

We start with the setup. 
Suppose that the $n$-dimensional stochastic process $\bx$ is governed by an OU process of the form,
\begin{equation}\label{Eq:WienerProcess}
    \frac{d}{dt}\bx = \bA\bx + \sqrt{2\bQ} \bmxi,
\end{equation}
where $\bA \in \R^{n \times n}$ is the dynamical matrix, $\bQ  \in \R^{n \times n}$ is the diffusion matrix, and $\bmxi \in \R^n$ is the normalized Gaussian random vector with zero mean and
\begin{equation}\label{Eq:NormalizedGaussian}
    \langle \bmxi(t+s)\bmxi(t)^T \rangle = \delta(s)\bI,
\end{equation} 
where the bracket and $\delta(s)$ denote the expectation and the Dirac delta function, respectively.
The diffusion matrix $\bQ$ is positively definite (i.e., symmetric and all eigenvalues being strictly positive), and to have a steady state, each eigenvalue of the dynamical matrix $\bA$ is assumed to have a negative real part. 
From now on, we assume the stationarity throughout this article.

\begin{theorem}
    Let the correlation function be given by $\bK(s) \coloneqq \langle \bx(\cdot+s) \bx(\cdot)^T \rangle$. For the linear stochastic process satisfying Eq.~(\ref{Eq:WienerProcess}), the dynamical matrix $\bA$ satisfies
    \begin{equation}
    \label{Eq:GreenFunc}
        \bA = \frac{1}{s} \log\big( \bK(s) \bK(0)^{-1} \big)
    \end{equation}
    for any $s>0$, where the $\log$ denotes the matrix logarithm. In particular, the correlation function is the exponential function of the form
    \begin{align} \label{Eq:Kw}
        \bK(s) = e^{\bA s} \bK(0) = e^{\bA s} \bCxx,
    \end{align}
    for $s \ge 0$, where $\bCxx \coloneqq \langle \bx(\cdot) \bx(\cdot)^T \rangle$ denotes the covariance of $\bx$.
\end{theorem}

\begin{proof}
    The transition matrix $\Psi(t,t')$ is simply defined by
    \begin{equation} \label{Eq:TransMat}
        \frac{\partial}{\partial t} \Psi(t,t') = \bA \Psi(t,t'),
    \end{equation}
    with initial condition $\Psi(t,t) = \bI$. Hence, the transition matrix can be explicitly written as 
    \[ \Psi(t,t') = e^{\bA(t-t')}. \]
    For a time-lag $s > 0$, by multiplying the transition matrix to Eq.~(\ref{Eq:WienerProcess}), rearranging and integrating, $\bx(t+s)$ can be expressed in terms of the transition matrix by 
    \[ \bx(t+s) = \Psi(t+s,t) \bx(t) + \int_{0}^s \Psi(t+t',t) \bmxi(t+t') dt'. \]
    We have
    \begin{widetext}
        \begin{align*}
        \langle \bx(t+s) \bx(t)^T \rangle &= \Psi(t+s,t) \langle \bx(t) \bx(t)^T \rangle + \int_0^s \Psi(t+t',t) \langle \bmxi(t+t')\bx(t)^T \rangle dt' \\
        &= e^{\bA s} \langle \bx(t) \bx(t)^T \rangle + \int_0^s \Psi(t+t',t) \langle \bmxi(t+t')\bx(t)^T \rangle dt' \\
        &= e^{\bA s} \langle \bx(t) \bx(t)^T \rangle,
    \end{align*}
    \end{widetext}
    since the Gaussian noise $\bmxi$ is independent of the stochastic process $\bx(t)$ and has zero mean. 
    Then in the steady state, a simple rearrangement completes the proof.
\end{proof}

\subsection{The fluctuation-dissipation relation (FDR)}
The FDR is a consequence of the Fokker-Planck equation. 
By characterizing the time evolution of probability distribution, the Fokker-Planck equation is used for analyzing and predicting the behavior of stochastic processes. 
The proof can be found in the standard SDE textbooks \cite{sarkka_solin_2019}.

\begin{theorem}[Classical Fokker-Plank equation]
    With the notation as above, the probability distribution $P(x,t) = \langle \delta(\bx(t)-x)\rangle$ of the state variable $\bx$ satisfies
    \begin{align} \label{Eq:FP-eq}
        \frac{\partial}{\partial t} &P(x,t) \nonumber \\
        &= - \sum_{j,k} \bA_{jk} \frac{\partial}{\partial x_j} x_k P(x,t) + \sum_{j,k} \bQ_{jk} \frac{\partial^2}{\partial x_j \partial x_k} P(x,t) \nonumber \\ 
        &= \LFP P(x,t),
    \end{align}
    where $\LFP$ is the Fokker-Planck operator. 
\end{theorem}

Now, we are in the position to prove the FDR by using the adjoint Fokker-Planck operator. 

\begin{corollary}[The classical fluctuation-dissipation relation] \label{Cor:FD-relation}
The covariance matrix $\bC_{\bx\bx}$ satisfies
    \begin{equation} \label{Eq:FD-relation}
        0 = \bA\bC_{\bx\bx} + \bC_{\bx\bx} \bA^T + 2 \bQ.
    \end{equation}
\end{corollary}

\begin{proof}
    The adjoint Fokker-Planck operator $\LFP^*$ is
    \begin{align*}
        \LFP^* = \sum_{j,k} \bA_{jk} x_k \frac{\partial}{\partial x_j} + \sum_{j,k} \bQ_{jk} \frac{\partial^2}{\partial x_j \partial x_k}.
    \end{align*}
    Let $\pst(x)$ be the stationary probability distribution (i.e., $\LFP \pst = 0$).
    Then we compute for each $p$ and $q$,
    \begin{align*}
        0 &= \int_{\R^n} x_p x_q \LFP \pst(x) \,dx\\
        &= \int_{\R^n} \pst(x) \LFP^* x_p x_q \,dx \\
        &= \int_{\R^n} \pst(x) \big( \sum_{j,k} \bA_{jk} x_k \delta_{jp} x_q + \sum_{j,k} \bA_{jk} x_k \delta_{jq} x_p \\ & \indent + \sum_{j,k} \bQ_{jk} \delta_{jp}\delta_{kq} \big) \,dx \\
        &= \sum_{k} \bA_{pk}\langle \bx_k \bx_q \rangle + \sum_{k} \bA_{jk}\langle \bx_k \bx_p \rangle + 2\bQ_{pq},
    \end{align*}
    where the $\delta_{jq}$ denotes the Kronecker delta.
    This is Eq.~(\ref{Eq:FD-relation}) in matrix notation. 
\end{proof}

The FDR links the linear deterministic dynamics $\bA$ and the diffusion matrix $\bQ$, and serves as one of the fundamental pieces of the Classical LIM.


\subsection{The algorithm of the Classical LIM}

In practice, given a time-series observation data $\{x(t): t = 0, \Delta t, \dots, N\Delta t\} \subset \R^n$, the Classical LIM returns the dynamical and diffusion matrices from the data based on (\ref{Eq:GreenFunc}) and (\ref{Eq:FD-relation}). 
Detrending $\{x(t)\}$ may be necessary prior to solving the system dynamics through LIMs.
However, the choice of detrend algorithm depends on applications and is beyond the scope of this article, so we do not dive into the details.
To compute the observed correlation function $\Kobs$, we use the approximation, for $s = k\Delta t$,
\begin{equation} \label{Eq:DiscCorrFuc}
    \Kobs(s) = \overline{x(t+s)x(t)^T} = \frac{\sum_{t=0}^{(N-k)\Delta t} x(t+s)x(t)^T}{N-k+1},
\end{equation} 
where the bar denotes the time average, as in the steady state, the ensemble average is equal to the time average. The Classical LIM is summarized in Alg.~\ref{Alg:LIM}.

\begin{algorithm}[H]
\caption{The Classical LIM} \label{Alg:LIM}
\begin{algorithmic}[1]
\Input{A time-series data $\{x(t)\}$, timestep $\Delta t$ and a lag $k$}
\Output{$\Alim$ \text{and} $\Qlim$}
    \State Detrend (Optional).
    \State Let the time-lag $\rho = k\Delta t$.
    \State Compute the correlation $\Kobs(0)$ and $\Kobs(\rho)$.
    \State Compute $\Alim$ by Eq.~(\ref{Eq:GreenFunc}).
    \State Compute $\Qlim$ by Eq.~(\ref{Eq:FD-relation}).
\end{algorithmic}
\end{algorithm}


\section{The Mathematics of Colored-LIM} \label{Chap:Colored-LIM}

In this section, we develop a novel LIM that models the environmental stochasticity with an OU colored noise. 
Based on the experience of the Classical LIM, if the noise correlation time $\bmtau$ is known, a natural approach is to find the dynamical matrix first and then the diffusion matrix. 
One of the main difficulties is that the white noise cannot be simply replaced by colored noise since the observable $\bx$ and the colored noise $\bmeta$ are correlated.
Therefore, we regard the colored noise as a stochastic process, and study both $\bx$ and $\bmeta$ at the same time in section \ref{Chap:Colored-LIM-Math}.
We leave to section \ref{Chap:Colored-LIM-algorithm} the discussion of the case where $\bmtau$ is not known a prior. 

\subsection{Mathematical background} \label{Chap:Colored-LIM-Math}

Similar to the case of Classical LIM, we study the correlation function of a colored-noise-driven linear system and the Fokker-Planck equation, and then use the adjoint Fokker-Planck operator to provide a recipe to estimate the linear dynamics and random forcing.

Again, we start with the setup. Let $\bx$ be the $n$-dimensional stochastic process satisfying
\begin{equation} \label{Eq:CG-StochasticProcess}
    \frac{d}{dt}\bx = \bA\bx + \sqrt{2\bQ}\bmeta,
\end{equation}
where $\bA \in \R^{n \times n}$ is the deterministic dynamical matrix and $\bQ \in \R^{n \times n}$ is the diffusion matrix; the normalized colored noise process $\bmeta$ is the an OU process satisfying
\begin{equation} \label{Eq:CG-Process}
    \frac{d}{dt}\bmeta = -\frac{1}{\bmtau} \bmeta + \frac{1}{\bmtau} \bmxi,
\end{equation} 
where $\bmtau >0$ is the noise correlation time, and $\bmxi$ is the $n$-dimensional normalized Gaussian vector satisfying Eq.~(\ref{Eq:NormalizedGaussian}).

The linear systems (\ref{Eq:CG-StochasticProcess}) and (\ref{Eq:CG-Process}) can be rearranged into the augmented stochastic equation
\begin{equation} \label{Eq:CG-and-SP-combined}
    \frac{d}{dt}
    \begin{bmatrix}
        \bx \\
        \bmeta
    \end{bmatrix}
    = 
    \begin{bmatrix}
        \bA &  \sqrt{2\bQ} \\
        0 & -\frac{1}{\bmtau}\bI 
    \end{bmatrix}
    \begin{bmatrix}
        \bx \\
        \bmeta
    \end{bmatrix}
    +
    \begin{bmatrix}
        0 & 0 \\
        0 & \frac{1}{\bmtau}\bI 
    \end{bmatrix}
    \begin{bmatrix}
        0 \\
        \bmxi 
    \end{bmatrix}
    .
\end{equation}
Therefore, the covariance matrices for the augmented system (\ref{Eq:CG-and-SP-combined}) satisfy
\begin{equation} \label{Eq:Cov-x-eta}
    \bC_{\bmeta\bx} \coloneqq \langle \bmeta(\cdot)\bx(\cdot)^T \rangle = \sqrt{\frac{\bQ}{2}}\bB^T
\end{equation}
and 
\[ \bC_{\bmeta\bmeta} \coloneqq \langle \bmeta(\cdot)\bmeta(\cdot)^T \rangle = \frac{1}{2\bmtau}\bI, \]
where $\bB = (\bI - \bmtau\bA)^{-1}$ is the well-defined resolvent of $\bA$ since it does not have any non-negative eigenvalue.
As $\bmeta$ is not available in practice, we study the correlation function $\bK(s) \coloneqq \langle \bx(\cdot+s) \bx(\cdot)^T \rangle$ of the state variable which can be written as, for $s \ge 0$,
\begin{align} \label{Eq:Kc}
    \bK(s) = e^{s\bA}\bCxx + e^{s \bA} \int_0^s e^{-s'(\bA+\frac{1}{\bmtau})} ds' \, \bQ\bB^T
\end{align} 
The first term is the exponential decay owing to deterministic dynamics, and the second term is the contribution from the colored noise. 
Since Eq.~(\ref{Eq:Kc}) is not a pure exponential function as Eq.~(\ref{Eq:Kw}), it is clear that Eq.~(\ref{Eq:GreenFunc}) is ineffective for a colored-noise-driven process.

We note that, unlike the case of Classical LIM, in the current setup, the dynamical term $\bA$ cannot be easily estimated by $2$ or multiple points of $\bK$; at least the inverse problem cannot be formulated into elementary functions.
Nevertheless, provided that $\bmtau$ is known, the estimation of $\bA$, and even $\bQ$, can be written as a linear problem by using the higher derivatives of $\bK$ at the origin, which we now derive.



\begin{theorem} \label{Thm:Colored-LIM}
    The generalized fluctuation-dissipation relation reads 
    \begin{equation} \label{Eq:LIM-CG-CorrFuc0}
        0 = \bA \bC_{\bx\bx} + \bC_{\bx\bx} \bA^T + \bQ\bB^T + \bB \bQ.
    \end{equation}
    The first, second, and third derivatives of the correlation function at $s = 0$ satisfy
    \begin{equation} \label{Eq:LIM-CG-CorrFuc1}
        \bK'(0) = \frac{1}{2} \big( \bA\bC_{\bx\bx} - \bC_{\bx\bx}\bA^T + \bQ\bB^T - \bB\bQ \big),
    \end{equation}
    \begin{equation} \label{Eq:LIM-CG-CorrFuc2}
        \bK''(0) = \frac{1}{2}\big( \bA(\bK'(0) + \frac{1}{\bmtau} \bC_{\bx\bx}) + (\bK'(0)^T + \frac{1}{\bmtau} \bC_{\bx\bx})\bA^T \big).
    \end{equation}
    and 
    \begin{align}\label{Eq:LIM-CG-CorrFuc3}
        \bK'''(0) &= \frac{1}{2}\big( \frac{1}{\bmtau^2}\bK'(0) + \bA(\bK''(0) - \frac{1}{\bmtau^2} \bC_{\bx\bx}) \nonumber \\
        & \indent - \frac{1}{\bmtau^2}\bK'(0)^T - (\bK''(0) - \frac{1}{\bmtau^2} \bC_{\bx\bx}) \bA^T \big).
    \end{align} 
    In addition, $\bK'(0)$ and $\bK'''(0)$ are skew-symmetric and $\bK''(0)$ is negatively definite. 
    Therefore, the systems of equations (\ref{Eq:LIM-CG-CorrFuc0}), (\ref{Eq:LIM-CG-CorrFuc2}), and (\ref{Eq:LIM-CG-CorrFuc3}) characterizes $\bA$ and $\bQ$. In particular, if $\bA$ is negatively definite, then Eqs.~(\ref{Eq:LIM-CG-CorrFuc0}) and (\ref{Eq:LIM-CG-CorrFuc2}) are sufficient. 
\end{theorem}

\begin{proof}

The Fokker-Planck equation for the probability distribution $P(x,\eta,t)$ of Eq.~(\ref{Eq:CG-and-SP-combined}) reads 
\begin{align*}
    \frac{\partial}{\partial t} P 
    &= -\sum_{i,j} \bA_{ij}\frac{\partial}{\partial x_i} x_j P - \sum_{i,j} \frac{\partial}{\partial x_i} (\sqrt{2\bQ})_{ij} \eta_{j}P \\ & \indent + \frac{1}{\bmtau} \sum_{i}\frac{\partial}{\partial \eta_i} \eta_i P + \frac{1}{2\bmtau^2} \sum_{i,j} \frac{\partial^2}{\partial \eta_i \partial \eta_j} P \\
    &= \LFP P,
\end{align*}
and the adjoint Fokker-Planck operator $\LFP^*$ is 
\begin{align*}
    \LFP^* &= \sum_{i,j} \bA_{ij}x_j\frac{\partial}{\partial x_i} + \sum_{i,j} (\sqrt{2\bQ})_{ij} \eta_j \frac{\partial}{\partial x_i} \\ & \indent - \frac{1}{\bmtau} \sum_{i}\eta_i\frac{\partial}{\partial \eta_i} + \frac{1}{2\bmtau^2}\sum_{i,j}\frac{\partial^2}{\partial \eta_i \partial \eta_j}.
\end{align*}
    
A direct computation shows that 
\begin{align*}
    \LFP^* \eta_p &= \frac{-1}{\bmtau} \eta_p,
    \intertext{and}
    \LFP^* x_p &= \sum_{j} \bA_{pj} x_j + \sum_{j} \eta_j (\sqrt{2\bQ})_{pj}. 
\end{align*}

By using a similar technique as in the proof of Corollary~\ref{Cor:FD-relation}, we have 
\begin{align*}
    0 &= \int_{\R^{2n}} x_p x_q \LFP \pst \,dx \,d\eta \\
    &= \sum_j \bA_{pj} \langle \bx_j \bx_q \rangle + \sum_j \bA_{qj} \langle \bx_j \bx_p \rangle \\
    & \indent + \sum_j (\sqrt{2\bQ})_{pj} \langle \bmeta_j \bx_q \rangle + \sum_j (\sqrt{2\bQ})_{qj} \langle \bmeta_j \bx_p \rangle.
\end{align*}
Applying Eq.~(\ref{Eq:Cov-x-eta}), we obtain, in matrix notation,
\begin{equation*} 
    0 = \bA \bC_{\bx\bx} + \bC_{\bx\bx} \bA^T + \bQ\bB^T + \bB \bQ.
\end{equation*}

Next, we compute the derivatives of the correlation function \cite{Jung1985,Risken1989}.
\begin{widetext}
   \begin{align*}
    \bK_{pq}(s) &= \int_{\R^{2n}} x_p e^{s \cdot {\LFP}}\pst x_q \,dx \,d\eta\\
    &= \int_{\R^{2n}} x_q \pst e^{s\cdot \LFP^*} x_p \,dx \,d\eta\\
    &= \int_{\R^{2n}} x_q \pst (\bI + s\LFP^* + \frac{s^2}{2} {\LFP^*}^2 + \frac{s^3}{6} {\LFP^*}^3) x_p + O(s^4) \,dx \,d\eta\\
    &= \langle \bx_p \bx_q \rangle + s \int_{\R^{2n}} x_q \pst \LFP^* x_p \,dx \,d\eta + \frac{s^2}{2} \int_{\R^{2n}} x_q \pst {\LFP^*}^2 x_p \,dx \,d\eta + \frac{s^3}{6} \int_{\R^{2n}} x_q \pst {\LFP^*}^3 x_p \,dx \,d\eta + O(s^4),
\end{align*} 
\end{widetext}
where $\pst(x,\eta)$ is the stationary probability distribution (i.e., $\LFP\pst = 0$).
By a direct computation, the first-order term is 
\begin{equation*}
    \sum_j \bA_{pj} \langle \bx_j \bx_q \rangle + \sum_j (\sqrt{2\bQ})_{pj} \langle \bmeta_j \bx_q \rangle.
\end{equation*}
Hence, we have, in the matrix notation,
\begin{equation*} 
    \bK'(0) = \bA\bC_{\bx\bx} + \bQ\bB^T.
\end{equation*}
By Eq.~(\ref{Eq:LIM-CG-CorrFuc0}), we see that $\bK'(0)$ is skew-symmetric and
\begin{align*}
    \bK'(0) &= \frac{1}{2}\big(\bK'(0) - \bK'(0)^T\big) \\
    &= \frac{1}{2} \big( \bA\bC_{\bx\bx} - \bC_{\bx\bx}\bA^T + \bQ\bB^T - \bB\bQ \big).
\end{align*}

By a similar computation, the second-order term becomes
\begin{align*}
    \bK''(0) &= \bA\bA \bC_{\bx\bx} + \bA\bQ\bB^T - \frac{1}{\bmtau} \bQ\bB^T \\
    &= \bA \bK'(0) - \frac{1}{\bmtau} \bQ\bB^T.
\end{align*}
To see $\bK''(0)$ is symmetric, we compute 
\begin{align*}
    \bA\bA &\bC_{\bx\bx} + \bA\bQ\bB^T - \frac{1}{\bmtau} \bQ\bB^T \\
    &= -\bA \bC_{\bx\bx} \bA^T -\bA \bB \bQ  - \frac{1}{\bmtau} \bQ\bB^T \\
    &= -\bA \bC_{\bx\bx} \bA^T - \bA \bB \bQ + \frac{1}{\bmtau} \bB\bQ - \frac{1}{\bmtau} (\bB\bQ + \bQ\bB^T) \\
    &= -\bA \bC_{\bx\bx} \bA^T + \frac{1}{\bmtau} \bQ - \frac{1}{\bmtau} ( \bB\bQ + \bQ\bB^T ).
\end{align*}
Since $\bK''(0)$ is symmetric, we obtain
\begin{align*}
    \bK''(0) &= \frac{1}{2}\big( \bK''(0) + \bK''(0)^T \big) \\
    &= \frac{1}{2} \big( \bA(\bK'(0) + \frac{1}{\bmtau} \bC_{\bx\bx}) + (\bK'(0)^T + \frac{1}{\bmtau} \bC_{\bx\bx})\bA^T \big).
\end{align*}

Moreover, repeating the argument for the third-order term, we have
\begin{equation*}
    \bK'''(0) = \bA^3\bC_{\bx\bx} + \bA^2\bQ\bB^T - \frac{1}{\bmtau}\bA\bQ\bB^T + \frac{1}{\bmtau^2}\bQ\bB^T - \frac{1}{\bmtau^2}\bQ.
\end{equation*}
Again, a similar argument shows that $\bK'''(0)$ is skew-symmetric, and we complete the proof.
\end{proof}

Theorem~\ref{Thm:Colored-LIM} is the core of the Colored-LIM algorithm.
It is now clear that given an observation data $\{ x(t) \}$, after a potential detrending process, the Colored-LIM computes the dynamical term $\bA$ by solving linear equations (\ref{Eq:LIM-CG-CorrFuc2}) and (\ref{Eq:LIM-CG-CorrFuc3}) and then the diffusion term $\bQ$ by the generalized FDR (\ref{Eq:LIM-CG-CorrFuc0}), provided the noise correlation time $\bmtau$ is known. Before addressing the issue of determining $\bmtau$, we discuss the difference between Classical-LIM and Colored-LIM.

In the $1$-dimensional case, we have $\bK'(0) = 0$ and
    \begin{equation} \label{Eq:1d-case}
        \frac{\bK''(0)}{\bC_{\bx\bx}} = \frac{\bA}{\bmtau}.
    \end{equation}
The derivatives of the correlation function reveal one of the fundamental differences between the white-noise-driven and colored-noise-driven processes. 
The correlation time near $s = 0$ is sharp and not (two-sided) differentiable at the origin for a white-noise-driven process while it is smoother for a colored-noise-driven one. 

In the white noise limit $\bmtau \to 0$, the Colored-LIM does not reduce to the Classical LIM since this limit process does not pass through the derivative. 
In fact, in a $1$-dimensional case, the correlation function of a colored-noise driven process becomes more concentrated at the origin while keeping its vanishing first derivative as $\bmtau$ decreases, but the one-sided right derivative of the correlation function in a white-noise case is $\bA\bC_{\bx\bx} \ne 0$.
Likewise, the higher order derivatives (\ref{Eq:LIM-CG-CorrFuc2}) and (\ref{Eq:LIM-CG-CorrFuc3}) do not converge to a white-noise case as $\bmtau \to 0$, indicating that the temporal structure of the colored noise changes the natural of the system. 

\subsection{The determination of the noise correlation time} \label{Chap:Colored-LIM-algorithm}

If the noise correlation time $\tau$ is known a prior, then in the previous subsection, we have shown that the dynamical and diffusion terms can be obtained by solving linear equations. 
However, in practice, $\tau$ is often unknown or is a target for studying, indicating that we either make an assumption on $\tau$ or require an empirical method to determine $\tau$ from observation. 
Here, we utilize the observed correlation function $\Kobs$ to propose a $\tau$-selection algorithm by the minimization
\begin{align} \label{Eq:Minimization}
    \taucglim = \argmin_{\tau} \left\Vert (K-\Kobs)|_{[0,l]} \right\Vert_F
\end{align} 
where $l$ is a free variable that represents the window length and $\left\Vert \cdot \right\Vert_F$ denotes the Frobenius norm.
More precisely, for each candidate $\tau$, the linear dynamics $A(\tau)$ and diffusion $Q(\tau)$ are solved as shown in the previous section; then the theoretical correlation function $K = K(A(\tau),Q(\tau),\tau)$ is computed by Eq.~(\ref{Eq:Kc}) and compared with the observed correlation function $\Kobs$. 
The model output $\taucglim$ is chosen to be the minimizer, and the corresponding linear dynamics and diffusion are the desired output. 
Finally, we note that the $\tau$-selection is a $1$-dimensional problem that can be done even by brutal force, and the condition $l \ge 2 \Delta t$ is required to avoid multiple minimizers.
The full Colored-LIM algorithm is summarized in Alg.~\ref{Alg:LIM-CG}.

\begin{algorithm}[H]
\caption{Colored-LIM (The full algorithm)}\label{Alg:LIM-CG}
\begin{algorithmic}[1]
\Input{A time-series data $\{x(t)\}$ and timestep $\Delta t$}
\Output{$\Acglim$\text{, } $\Qcglim$ \text{, and} $\taucglim$}
\State Detrend (optional).
\State Compute the correlation function and its derivatives $\Kobs^{(m)}(0)$ for $m = 0,\dots,3$.
\For{candidate $\tau$}
    \State Compute $A(\tau)$ by Eqs.~(\ref{Eq:LIM-CG-CorrFuc2}) and (\ref{Eq:LIM-CG-CorrFuc3}).
    \State Compute $Q(\tau)$ by Eq.~(\ref{Eq:LIM-CG-CorrFuc0}).
    \State Compute the correlation function $K$ for the linear system by Eq.~(\ref{Eq:Kc}).
    \State Compute the error shown in Eq.~(\ref{Eq:Minimization}).
\EndFor
\State Select $\tau$ that minimizes the error and set it to be $\taucglim$.
\State Select $\Acglim$ and $\Qcglim$ corresponding to $\taucglim$.
\end{algorithmic}
\end{algorithm}

\subsection{From an application perspective} \label{Chap:Application-viewpoint}

The Colored-LIM can be applied to a wide range of problems. Here, we briefly discuss its potential applications.

\begin{itemize}
    \item For an autonomous system with additive noise of the form 
    \begin{align} \label{Eq:auto-system}
        \frac{d}{dt}\bx = f(\bx,\mu) + \textit{noise}
    \end{align} 
    where the noise term can be either white or colored noise, the LIMs estimate the Jacobian $\Jf$ of the dynamics $f$.
    If the non-linear dynamics $f$ is mathematically tractable, then $\Jf$ may provide further insight into the underlying system. 
    See sections~\ref{Chap:Network-identification} and \ref{Chap:ENSO} for examples. 
    \item In real-world applications, LIMs are often applied to lower-ranked representations of higher-dimensional data for two key reasons.
    First, the dominant modes of the data typically capture the essential physical phenomena of interest, and the projected dynamics and diffusion describe the interactions between these modes. 
    Second, applying LIMs directly to higher-dimensional time-series data can lead to numerical instability, particularly when the observation window is limited. 
    These issues are mitigated by reducing dimensionality, giving rise to more interpretable and robust results. 
    Typical methods for mode identification and dimensionality reduction include principal component analysis (PCA).
    See section~\ref{Chap:DMD} for instance.    
\end{itemize}

\subsection{Connections with dynamic mode decomposition} \label{Chap:Conn-with-DMD}

DMD is a data-driven method that focuses on extracting spatiotemporal patterns from observation data without requiring a priori knowledge of the underlying system equations.
Though the DMD framework covers a wide range of variants, the \textit{projected} DMD and the Classical LIM are equivalent algorithms under certain conditions \cite{Tu2014}. 
To see the connection between the Colored-LIM and the projected DMD, we start by briefly introducing the DMD. For a more thorough discussion, we refer the readers to \cite{Kutz2016,Schmid2022,Tu2014}.

The DMD adopts an equation-free perspective as LIMs do and estimates the unknown dynamics of a complex system 
\begin{align} \label{Eq:General-DMD-eq}
    \frac{d}{dt}\bx = f(\bx,t,\mu)
\end{align}
by producing an approximate local linear system of the form
\begin{align} \label{Eq:DMD-eq}
    \frac{d}{dt}\bx = \bA \bx
\end{align} 
with initial condition $\bx(0)$.
For a continuous process $\bx$ satisfying Eq.~(\ref{Eq:DMD-eq}), we have 
\begin{align} \label{Eq:DMD-discrete}
    \bx(t+\Delta t) = \bL \bx(t)
\end{align}
where $\bL = e^{\bA\Delta t}$. 
Eq.~(\ref{Eq:DMD-discrete}) is the key ingredient of the DMD.
Given a discrete time-series $\{x_j = x(j \Delta t)\}$, the projected DMD represents the system dynamics through $Y = \Ldmd X$ where 
\begin{align*}
    Y = \begin{bmatrix}
    | &  & | \\
    x_1 & \cdots & x_N \\
    | &  & |
\end{bmatrix}, 
\text{ and }
    X = \begin{bmatrix}
    | &  & | \\
    x_0 & \cdots & x_{N-1} \\
    | &  & |
\end{bmatrix},
\end{align*}
and extracts it by linear regression followed by matrix logarithm,
\begin{align} \label{Eq:DMD-regression}
    \Admd = \frac{1}{\Delta t}\log(Y X^\dag),
\end{align}
where the dagger $\dag$ denotes the pseudo-inverse.

The projected DMD framework differs from the Classical LIM framework in several aspects. 
Firstly, the DMD considers a noise-free dynamical system (except for a potential measurement noise). 
Hence, a limited sampling size is allowed by DMD, while for LIMs, sufficient sampling is required to average out the random perturbation so that the observed correlation function decently represents the underlying system. 
Moreover, DMD does not require the system state to be close to an equilibrium point; hence, DMD is valid for input data with or without mean removed. 
However, without the mean removed, the dynamical matrix estimated by DMD cannot be interpreted by the Jacobian of $f$.

When the mean is removed, the projected DMD is equivalent to the Classical LIM in the sense that the estimated dynamical matrices are identical. 
The proof can be found in \cite{Tu2014}, so here we simply give a heuristic argument. 
The regression process applied to the pseudo-inverse in Eq.~(\ref{Eq:DMD-regression}) assumes independent and identically distributed (IID) white noise residuals that correspond to the noise term in Eq.~(\ref{Eq:WienerProcess}).
The source of the residuals in the projected DMD framework is interpreted by the non-linearity and measurement error, but in the Classical LIM framework, furthermore interpreted by white noise random forcing.
The equivalence can be seen from a regression perspective as well.
In the Classical LIM framework, multiple linear regression (MLR) can be used in the equation 
\begin{align} \label{Eq:MLR}
    x_{j+1} - x_j = \Alim \, x_j \Delta t + \xi \Delta t
\end{align} 
to estimate the linear dynamics instead of the method of transition matrix introduced in section~\ref{Chap:TransFunc} due to the IID white noise term.
Therefore, one now easily see that Eq.~(\ref{Eq:MLR}) results in the generator of $\Ldmd$ in Eq.~(\ref{Eq:DMD-discrete}).

In the Colored-LIM framework, the correlation between the state variable and colored noise (\ref{Eq:Cov-x-eta}) leads to non-IID residual terms from a regression perspective.
Hence, the MLR approach or other simple regression methods is not valid for estimating the linear dynamics of a colored-noise-driven process.  
Nevertheless, the Colored-LIM is developed by circumventing the non-trivial correlation from an SDE perspective so that we may say that from a regression perspective, the Colored-LIM allows the residuals to depend on themselves and the observation data.  
Therefore, as a new variant of the Classical LIM, the Colored-LIM can be viewed as an extension of DMD, and be utilized to find the dynamic mode of a colored-noise-driven stochastic system. 


\section{Validation} \label{Chap:NumExp}
In this section, we validate the effectiveness of Colored-LIM by ideal numerical studies. 
Eqs.~(\ref{Eq:LIM-CG-CorrFuc0}), (\ref{Eq:LIM-CG-CorrFuc2}), and (\ref{Eq:LIM-CG-CorrFuc3}) are essentially based on infinite realizations of the linear system (\ref{Eq:CG-StochasticProcess}), but in practice, we only have a time-series observation data (i.e., one realization).
For each experiment, we generate the realization of Eq.~(\ref{Eq:CG-StochasticProcess}) by using the second-order Milstein approximation with an integration time step $0.001$ and time span $T$ \cite{Milshtein1979}, and then subsampling so that the observation $\{ x(t) \}$has equal sampling interval $\Delta t$.
Finally, we apply the Colored-LIM to estimate the dynamics and random forcing. 
For ideal numerical studies, no detrending is applied.

Before proceeding, we explain the terminology.
For terminological convenience, we call the linear system (\ref{Eq:CG-Process}) determined by the dynamics and random forcing output from Alg.~\ref{Alg:LIM-CG} the Colored-LIM process, or simply Colored-LIM, with an abuse of terminology whenever there is no confusion.
Also, we use Colored-LIM as an adjective to indicate, for example, the correlation functions determined by a Colored-LIM process. 
A similar terminology applies to the Classical LIM.

\subsection{Solving a linear problem}
To validate Eqs.~(\ref{Eq:LIM-CG-CorrFuc0}), (\ref{Eq:LIM-CG-CorrFuc2}), and (\ref{Eq:LIM-CG-CorrFuc3}), we consider the following linear problem
    \begin{align} \label{Eq:Linear_Problem}
    \frac{d}{dt}
    \begin{bmatrix}
        \bx_1 \\ \bx_2 \\ \bx_3 
    \end{bmatrix} 
    & =
    \begin{bmatrix*}[r]
        -1.0 & 0.5 & 1.0 \\ 0.5 & -2.0 & 0.0 \\ 0.0 & 1.0 & -1.0
    \end{bmatrix*} 
    \begin{bmatrix}
        \bx_1 \\ \bx_2 \\ \bx_3 
    \end{bmatrix} 
    \nonumber \\
    & \indent +
    \sqrt{ 2 \cdot
    \begin{bmatrix*}[r]
        1.0 & 0.0 & 0.5 \\ 0.0 & 2.0 & 1.0 \\ 0.5 & 1.0 & 2.0
    \end{bmatrix*}
    }
    \begin{bmatrix}
        \bmeta_1 \\ \bmeta_2 \\ \bmeta_3 
    \end{bmatrix} 
    ,
\end{align}
with $\bmtau = 0.5$ and realize the state variable $\bx$ with $T = 1000$, and $\Delta t = 0.1$.
To evaluate the performance of Colored-LIM, we use the relative error $e_{\bX}$ measured by the Frobenious norm. That is,
\[ e_{\bX} = \frac{\left\Vert X_\text{Model} - \bX \right\Vert_F}{\left\Vert \bX \right\Vert_F}, \]
where $X_\text{Model}$ is the model output and $\bX$ is the ground truth.

Given observation data, without any prior knowledge, a direct application of Colored-LIM with window length $l = 1$ returns
\begin{widetext}
   \begin{align} 
    \Acglim
    =
    \begin{bmatrix*}[r]
       -1.02 &  0.42 &  1.02 \\
        0.41 & -2.05 &  0.10 \\ 
       -0.02 &  0.80 & -1.00
    \end{bmatrix*}
    \text{, }    
    \Qcglim
    =
    \begin{bmatrix*}[r]
      1.10 & 0.00 & 0.59 \\
      0.00 & 1.92 & 0.99 \\
      0.59 & 0.99 & 2.18
    \end{bmatrix*}
    \text{, and }
    \taucglim = 0.53,
\end{align} 
\end{widetext}
whose relative errors $e_\bA$, $e_\bQ$, and $e_\bmtau$ are $9.7\%$, $7.9\%$, and $6.2\%$ respectively.
Indeed, the Colored-LIM mechanism can be visualized by Fig.~\ref{Fig:Linear_CorrFunc}, which shows the true correlation function $\bK$ of the linear system (\ref{Eq:Linear_Problem}), the observed correlation function $\Kobs$ based on the realization, and the Colored-LIM correlation function $K = K(\Acglim,\Qcglim,\taucglim)$.
First of all, $\bK$ and $\Kobs$ are not perfectly matched, with an absolute difference $E_0 = \left\Vert \bK - \Kobs \right\Vert_F = 0.520$, due to the numerical error introduced by the Milstein approximation which we cannot avoid. 
However, $\Kobs$ and $K$ are significantly more closely aligned, with an absolute difference $E_1 = \left\Vert K - \Kobs \right\Vert_F = 0.096$, significantly smaller than $E_0$. 
One can imagine that the Colored-LIM algorithm learns the underlying system by finding the best approximate linear system from $\Kobs$ without being aware of the true answer, or equivalently, by fitting the theoretical $K$ to $\Kobs$ over the search space $\{ \tau > 0 \}$.
Therefore, in this linear problem, the relative errors are indeed mostly due to the numerical error of Milstein approximation.

\begin{figure}[h]
    \centering
    \includegraphics[]{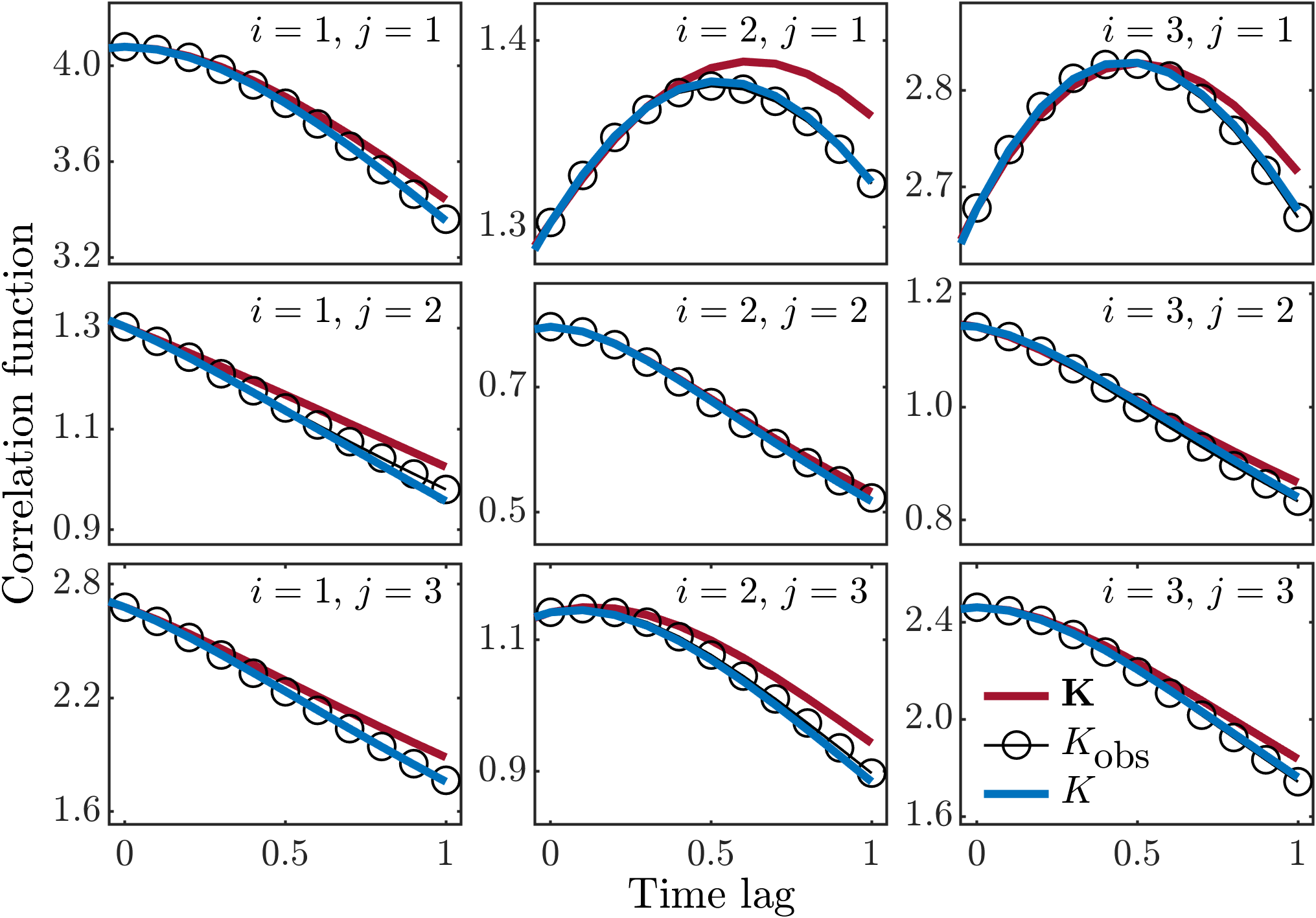}
    \caption{ The $(i,j)$-entry of the correlation functions where $i$ and $j$ index the rows and columns, respectively. Each subplot corresponds to a different combination of $i$-th and $j$-th variables. The ground truth $\bK$ for Eq.~(\ref{Eq:Linear_Problem}) is represented by the solid red line, the observation $\Kobs$ is marked by a black circle, and the one $K$ constructed by Colored-LIM is shown by blue. }
    \label{Fig:Linear_CorrFunc}
\end{figure}

To ensure that the discussion is not merely a special case or a rare incident, we repeat the above process over $1024$ independent realizations and summarize the statistics in Table~\ref{Table:LinearProblem}.
In addition, it is clear that as the time span $T$ increases to $5000$, the relative and absolute errors decrease. 
At the same time, as the numerical error of Milstein approximation decreases for a larger $T$, the difference between $E_0$ and $E_1$ becomes less significant.

\begin{table}[h]
\caption{\label{Table:LinearProblem}%
The statistics of relative and absolute errors under different time spans $T$ and noise correlation times $\bmtau$. For each condition, the errors are presented as mean $\pm$ standard deviation, with relative errors $e_\bullet$ expressed in percentages and absolute errors $E_\bullet$ being unitless.
}
\begin{ruledtabular}
\begin{tabular}{cccccc}
 & $e_\bA$ \fontsize{7pt}{7pt}\selectfont (\%) & $e_\bQ$ \fontsize{7pt}{7pt}\selectfont (\%) & $e_\bmtau$ \fontsize{7pt}{7pt}\selectfont (\%) & $E_0$ & $E_1$ \\
\midrule
    \fontsize{7pt}{7pt}\selectfont $\begin{array} {r@{}l@{}} T & {}= 1000 \\ \bmtau & {}= 0.5 \end{array}$ &
    \fontsize{7pt}{7pt}\selectfont $10.5_{\pm 4.0}$ &
    \fontsize{7pt}{7pt}\selectfont $8.6_{\pm 5.9}$ &
    \fontsize{7pt}{7pt}\selectfont $4.6_{\pm 3.6}$ &
    \fontsize{7pt}{7pt}\selectfont $0.3105_{\pm 0.1519}$ & 
    \fontsize{7pt}{7pt}\selectfont $0.1325_{\pm 0.0581}$ \\
\cmidrule(l{0.2em}r{0.2em}){1-6}
    \fontsize{7pt}{7pt}\selectfont $\begin{array} {r@{}l@{}} T & {}= 5000 \\ \bmtau & {}= 0.5 \end{array}$ & 
    \fontsize{7pt}{7pt}\selectfont $6.1_{\pm 1.8}$ &
    \fontsize{7pt}{7pt}\selectfont $4.5_{\pm 2.6}$ &
    \fontsize{7pt}{7pt}\selectfont $4.1_{\pm 2.2}$ & 
    \fontsize{7pt}{7pt}\selectfont $0.1225_{\pm 0.0717}$ &
    \fontsize{7pt}{7pt}\selectfont $0.1049_{\pm 0.0299}$ \\
\cmidrule(l{0.2em}r{0.2em}){1-6}
    \fontsize{7pt}{7pt}\selectfont $\begin{array} {r@{}l@{}} T & {}= 1000 \\ \bmtau & {}= 1.0 \end{array}$ &
    \fontsize{7pt}{7pt}\selectfont $16.2_{\pm 6.0}$ & 
    \fontsize{7pt}{7pt}\selectfont $22.5_{\pm 13.2}$ & 
    \fontsize{7pt}{7pt}\selectfont $10.4_{\pm 6.9}$ & 
    \fontsize{7pt}{7pt}\selectfont $0.1994_{\pm 0.1007}$ & 
    \fontsize{7pt}{7pt}\selectfont $0.0657_{\pm 0.0324}$ \\
\cmidrule(l{0.2em}r{0.2em}){1-6}
\fontsize{7pt}{7pt}\selectfont $\begin{array} {r@{}l@{}} T & {}= 5000 \\ \bmtau & {}= 1.0 \end{array}$ & 
\fontsize{7pt}{7pt}\selectfont $11.9_{\pm 3.6}$ &
\fontsize{7pt}{7pt}\selectfont $19.2_{\pm 7.1}$ & 
\fontsize{7pt}{7pt}\selectfont $8.5_{\pm 4.0}$ & 
\fontsize{7pt}{7pt}\selectfont $0.0891_{\pm 0.0489}$ &
\fontsize{7pt}{7pt}\selectfont $0.0394_{\pm 0.0150}$ \\
\end{tabular}
\end{ruledtabular}
\end{table}

Next, if the noise correlation time $\bmtau = 0.5$ is known, then the Colored-LIM returns the mean of $e_\bA$ and $e_\bQ$ to be $11.0\%$ and $7.6\%$, respectively, for $T = 1000$.
Here, we merely mean to demonstrate that the prior knowledge can be incorporated into the Colored-LIM algorithm instead of stating that the performance will improve, since again, the numerical error from the Milstein approximation is inevitable.

We have tested Eqs.~(\ref{Eq:LIM-CG-CorrFuc0}), (\ref{Eq:LIM-CG-CorrFuc2}), and (\ref{Eq:LIM-CG-CorrFuc3}) for various $\bA$, $\bQ$, $\bmtau$, $l$ and $\Delta t$. 
As long as the $T$ is sufficiently long and $\Delta t$ is sufficiently fine so that the observed correlation function well represents the underlying system, Eqs.~(\ref{Eq:LIM-CG-CorrFuc0}), (\ref{Eq:LIM-CG-CorrFuc2}), and (\ref{Eq:LIM-CG-CorrFuc3}) can estimate the underlying dynamics and random forcing accurately. 
However, we notice that in general, for a larger $\bmtau$, a longer observation is required since the sample path of $\bx$ is greatly smoothed owing to the strong noise memory, and it takes more time for $\bx$ to walk through the state space.
For example, we consider the linear problem (\ref{Eq:Linear_Problem}) again, with the noise correlation time set to $\bmtau = 1$, while keeping all other parameters and conditions the same as before.
Table~\ref{Table:LinearProblem} shows that the mean values of the relative errors are larger than the case of $\bmtau = 0.5$.

\subsection{Solving a non-linear problem: SIS model and network identification}\label{Chap:Network-identification}

Though Colored-LIM stems from linear approximation, we show that Eqs.~(\ref{Eq:LIM-CG-CorrFuc0}), (\ref{Eq:LIM-CG-CorrFuc2}), and (\ref{Eq:LIM-CG-CorrFuc3}) can be applied to non-linear problems to gain insightful information on the underlying complex systems.
In particular, we consider the ideal susceptible-infectious-susceptible (SIS) model.

The SIS model is a population-level model used in epidemiology and infectious disease modeling.
It contains two compartments for the susceptible and infectious populations with the infectious rate $\beta$ and the recovery rate $\gamma$, characterizing infectious diseases that do not provide lifelong immunity \cite{Bartesaghi2024}.
The scalar SIS model assumes that individuals are mixed homogeneously so that every individual has an equal likelihood of interacting with others, and is described by the differential equation
\begin{align} \label{Eq:Scaler-SIS}
    \frac{d\bx}{dt} = \beta (1-\bx)\bx - \gamma \bx.
\end{align} 
which admits a closed solution. 
It is well-known that a unique disease-free steady state exists if the basic reproductive ratio $R = \frac{\beta}{\gamma}$ is smaller than the epidemic threshold $R_s = 1$.
On the other hand, if $R>R_s$, there are $2$ equilibrium points: the unstable disease-free steady state and the exponentially stable endemic steady state with $\bx^s = 1 - \frac{\gamma}{\beta}$. 
In addition, a bifurcation occurs at $R = R_s$.
Due to the simplicity of Eq.~(\ref{Eq:Scaler-SIS}), its mathematical properties are widely studied, but in practice, it may be an over-idealization of the complex social and epidemic behaviors.

In this article, we consider a stochastic network-level SIS model in which each node represents a group of individuals like a city, or a country, the weighted edges describe potential interactions that could lead to disease transmission for instance, transportation network, and the noise models the perturbation such as variations in the incubation period among individuals \cite{Bartesaghi2024,Bonaccorsi2022,VanMieghem2012}.
More specifically, the governing equation is written as 
\begin{align} \label{Eq:Network-SIS}
    \frac{d\bx}{dt} &= f(\bx,[\bW,\beta,\gamma]) + \sqrt{2\bQ}\bmeta \nonumber \\ 
    &= \beta \big( \bI-\text{diag}(\bx) \big) \bW \bx - \gamma \bx + \sqrt{2\bQ}\bmeta.
\end{align} 
where $\bW$ denotes the weight matrix of a \textit{simple} graph (i.e., no multi-loop or self-loop), $\text{diag}(\bx)$ is a diagonal matrix whose diagonal are $\bx$, and $\bmeta$ is the OU colored noise of the form (\ref{Eq:CG-Process}).
The epidemic threshold $R_s$ depends on the topology of the network, and when the endemic steady state $\bx^s$ exists, the Jacobian $\Jf$ at $\bx^s$ reads
\begin{align}
    \Jf(\bx^s) = 
    \begin{cases}
        \beta ( 1 - \bx^s_i )\bW_{ij}, & \text{ if } i \ne j\\
        -\beta \sum_k \bW_{ik}\bx^s_k - \gamma, & \text{ if } i = j.
    \end{cases}
\end{align}
With the Jacobian, the relative magnitude $\widetilde{\bW}$ of weights on edges can be retrieved even with $\beta$ unknown (i.e., $\widetilde{\bW}_{ij} = \beta \, \bW_{ij}$ for $i \ne j$). 

We set up the SIS model (\ref{Eq:Network-SIS}) with $\beta = 1$, $\gamma = 2$, $\bmtau = 1$, and $\bQ = 5 \times 10^{-5}\,\bI$ where $\bQ$ is chosen such that the sample path has a high probability to stay in [0,1], and simple weight (i,e, $\bW_{ij} = 1$ if and only if node $i$ and $j$ are adjacent)

\begin{align*}
    \bW = \widetilde{\bW} = 
    \begin{bmatrix}
        0 & 1 & 1 & 1 \\
        1 & 0 & 1 & 0 \\
        1 & 1 & 0 & 0 \\
        1 & 0 & 0 & 0
    \end{bmatrix}, 
\end{align*}
which represents the network interactivity, as visually shown in Fig.~\ref{Fig:Network-SIS}. 
We demonstrate the effectiveness of Colored-LIM in the SIS model with or without prior knowledge by generating an ensemble of realization $\{ x(t) \}$ of the prevalence of infected individuals using the Euler method with time span $T = 3000$, integration time step $0.001$, and sampling interval $\Delta t = 0.1$.

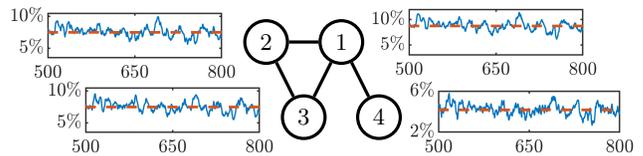
\begin{figure}[h]
    \centering
    \begin{tikzpicture}
    \node[mynode](n4) at (0,-0.5){4};
    \node at (1.9,-0.5) {\includegraphics{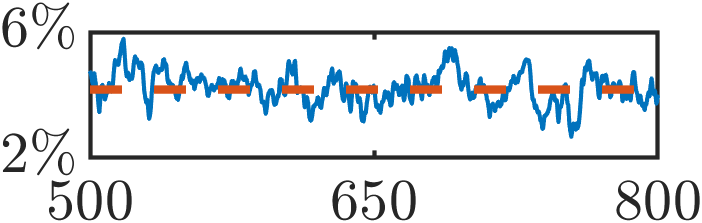}}; 
    \node[mynode](n3) at (-0.5,0.5){1}; 
    \node at (1.4,0.5) {\includegraphics{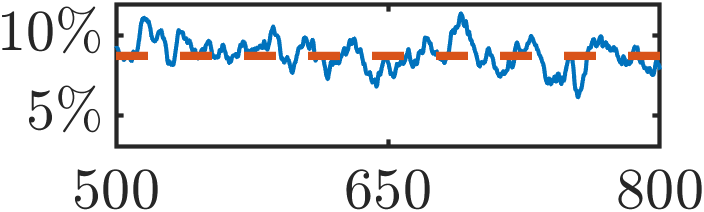}}; 
    \node[mynode](n2) at (-1.0,-0.5){3};
    \node at (-2.9,-0.5) {\includegraphics{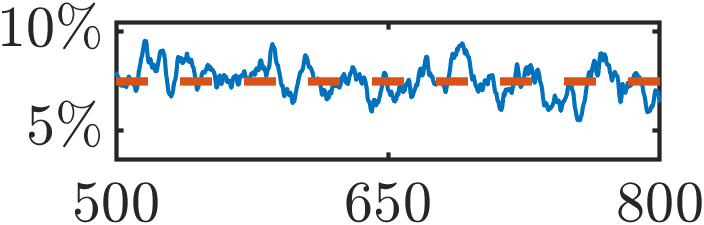}}; 
    \node[mynode](n1) at (-1.5,0.5){2};
    \node at (-3.4,0.5) {\includegraphics{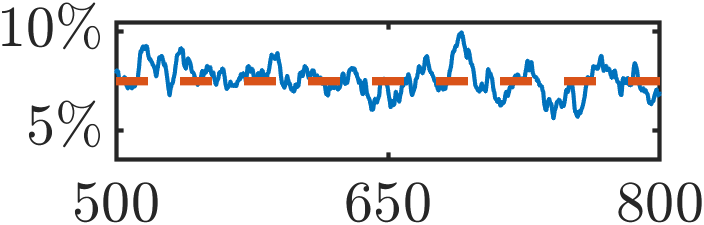}}; 
    \draw[myarrow](n1)--(n3);
    \draw[myarrow](n1)--(n2);
    \draw[myarrow](n2)--(n3);
    \draw[myarrow](n3)--(n4);
    \end{tikzpicture}
    \caption{The underlying network of the network-level SIS model in the numerical experiment in section~\ref{Chap:Network-identification}. The time series shown at each node is a segment of the prevalence of infected individuals (input data). The orange dashed line indicates the analytic mean value (steady state).}
    \label{Fig:Network-SIS}
\end{figure}

\textit{No prior knowledge.} The application of Colored-LIM to $\{x(t)\}$ returns a dynamical matrix $J$ that approximates the Jacobian matrix $\Jf$ and noise correlation time with mean values of $e_\Jf$ and $e_\bmtau$ being $9.0\%$ and $3.4\%$, respectively, which again, are mostly due to the Euler method instead of the linear approximation (\ref{Eq:CG-StochasticProcess}).
Then the relative magnitude of weights $\widetilde{W}_{ij}$ is identified by identified by the non-diagonal entries $\widetilde{W}_{ij} = \frac{J_{ij}}{1-x_i^s } \approx \beta \, \bW_{ij}$, where $x_i^s = \overline{ x_i(t) }$, which is
\begin{align*}
    \widetilde{W} = 
    \begin{bmatrix*}
        0.00_{\pm 0.00} & 0.97_{\pm 0.08} & 0.97_{\pm 0.08} & 0.96_{\pm 0.08} \\
        0.97_{\pm 0.08} & 0.00_{\pm 0.00} & 0.97_{\pm 0.08} & 0.02_{\pm 0.06} \\
        0.97_{\pm 0.08} & 0.97_{\pm 0.08} & 0.00_{\pm 0.00} & 0.02_{\pm 0.06} \\
        0.96_{\pm 0.08} & 0.02_{\pm 0.06} & 0.02_{\pm 0.06} & 0.00_{\pm 0.00} \\
    \end{bmatrix*}
\end{align*}
where the notation $\pm$ represents the standard deviation.
We see that the Colored-LIM successfully captures the topology of the network, where weights $\widetilde{W}_{ij}$ are highly concentrated on the correct entries.

\textit{$\tau$ is known.} 
Under this condition, the relative error $e_\Jf$ exhibits a mean of $9.5\%$, and the estimated weight matrix is as effective as the previous case. 
Again, due to the numerical error of the Euler method, there is no guarantee that the relative error is significantly lower in this ideal numerical study.

\section{Applications to Real-world Complex Systems} \label{Chap:Real-world}

In this section, the Colored-LIM is applied to real-world observation data, and compared with the Classical LIM.
We note that our purpose is not a complete investigation but a demonstration of the potential application and implication of Colored-LIM. 

\subsection{ENSO: prediction skill} \label{Chap:ENSO}
The El Ni\~{n}o-Southern Oscillation (ENSO) is the most influential interannual climate mode, driven by ocean-atmosphere interactions over the tropical Pacific Ocean \cite{Lorenzo2015}.
ENSO alternates between two distinct phases: El Ni\~{n}o and La Ni\~{n}a, characterized by anomalously warm and cool sea surface temperatures (SST) in the central and eastern Pacific, respectively \cite{Capotondi2015,Okumura2011}. 
These phases have significant global impacts, causing shifts in weather patterns that can lead to extreme events such as floods, droughts, and tropical storms, thereby increasing disaster risks worldwide \cite{Cai2015,Philip2014}. 
The Ni\~{n}o 3.4 index, defined by the SST anomaly averaged over the Ni\~{n}o 3.4 region (170W-120W, 5S-5N), is an indicator commonly used to monitor the phase and intensity of El Ni\~{n}o and La Ni\~{n}a.
Consequently, modeling the Ni\~{n}o 3.4 index is crucial for improving the prediction of these extreme events and enabling disaster mitigation strategies to minimize the impacts of ENSO-driven climate variability.

In climate science, small-scale processes such as turbulence or atmospheric events are often considered as external noise forcing in stochastic modeling \cite{Matthew2011,Penland1989,Penland1993}. 
One of the simplest lower-dimensional linear modelings for ENSO uses the anomaly of SST and the 20$^\circ$C isotherm depth (D20) over the Ni\~{n}o 3.4 region. 
Previous studies where the environmental stochasticity is implicitly modeled by memoryless white noise have suggested that the ENSO mechanism can be regarded as a recharge oscillator, where the SST anomaly and D20 anomaly play the roles of momentum and position, respectively \cite{Burgers2005,Jin2007}. 
Here, we re-examine this lower-dimensional representation of ENSO and model the random forcing by white and colored noise.
The data is obtained from the ORAS5 reanalysis from 1979 to 2022 and then is filtered by a $3$-month moving average, which is a common practice in ENSO studies \cite{Matthew2011,Zuo2019}. 
The state variable $\{ x(t) \} \subset \R^2$ is set to be the z-score of the filtered SST and D20 data indexed by $1$ and $2$, respectively; the time lag $\rho$ in the Classical LIM and the window length $l$ in the Colored-LIM is set to be $3$ and $12$ months, respectively.

The LIM results indicate that 
\begin{widetext}
\begin{subequations} \label{Eq:LIM_ENSO}
    \begin{gather} \label{Eq:LIM_ENSO_White}
    \Alim = 
    \begin{bmatrix*}[r]
    -0.0829 & 0.1406 \\
    -0.1629 & 0.0014
    \end{bmatrix*}
    , 
    \Qlim = 
    \begin{bmatrix*}[r]
    0.0440 & 0.0221 \\
    0.0221 & 0.0438
    \end{bmatrix*}
    ,
    \intertext{and}
    \Acglim = 
    \begin{bmatrix*}[r]
    -0.1178 &  0.1372 \\
    -0.1670 & -0.0262
    \end{bmatrix*}
    , 
    \Qcglim =
    \begin{bmatrix*}[r]
    0.0926 & 0.0440 \\
    0.0440 & 0.0916
    \end{bmatrix*}
    , 
    \taucglim = 1.8480, \label{Eq:LIM_ENSO_Colored}
    \end{gather}
\end{subequations}
\end{widetext}
where the time unit is in months.
In the case of Classical LIM, $\Alim$ is close to the Jacobian of the damped harmonic oscillator
\begin{align*}
    \begin{bmatrix*}[c]
    -2\gamma & \omega \\
    -\omega   & 0
    \end{bmatrix*}
\end{align*}
with a period of $2 \pi \omega^{-1} \approx 38.6 \sim 44.7$ months and a delay time $\gamma^{-1} \approx 24.1$ months, consistent with the recharge oscillator model \cite{Burgers2005}.
However, in the Colored-LIM framework, such a structure does not exist in $\Acglim$ since D20 damps itself with an unignorable rate of $-0.0262$ month$^{-1}$, not significantly lower than the other entries.
Moreover, the $1.85$-month $e$-folding time indicates that though random, the atmospheric processes modeled by colored noise $\eta$ may have an unignorable coupling effect (referred to as state dependency in some literature, e.g., \cite{Jin2007}) of 
\begin{align} 
    C_{\eta x}
    = \sqrt{\frac{\Qcglim}{2}}\Bcglim^T = 
    \begin{bmatrix*}[r]
    0.5933 & -0.0001 \\
    0.2761 &  0.6041
    \end{bmatrix*}
\end{align}
where $\Bcglim^T = (\bI - \taucglim\Acglim)^{-1}$,
and a continuous influence on oceanic processes represented by $x$ over months.
This suggests the importance of accounting for the temporal structure, namely the colored noise characteristics, of environmental stochasticity in mathematical modeling.

We regard the Colored-LIM and Classical LIM processes as lower-dimensional representations of the ENSO.
To compare them, we first observe the difference in their sample paths by integrating two LIMs with a $43$-year time span and a $1/12$-month time step and subsampling so that the $\Delta t = 1$ month.
Fig.~\ref{Fig:SamplePath} shows the observation and LIMs' trajectories. 
It seems that in the white-noise case, the trajectory oscillates faster than the observation, while the Colored-LIM trajectory looks more similar to the observation.
Indeed, the correlation functions shown in Fig.~\ref{Fig:CorrFunc} support this viewpoint. 
The Colored-LIM auto-correlations (diagonal entries), due to the skew-symmetry of $K'(0)$ and the negative definiteness of $K''(0)$, have gentler slopes with downward concavity near the origin, which aligns with the observed SST and D20 auto-correlations,
From a physical perspective, the noise memory smooths the trajectory, increasing the correlation between the current and near-future states.  
In contrast, in the Classical LIM framework, the steeper (one-sided) slope of the auto-correlations indicates a more intense oscillation in trajectory, and the (one-sided) concavity though downward in this case, is not guaranteed to be in general.

\begin{figure}[h]
    \centering
    \includegraphics[]{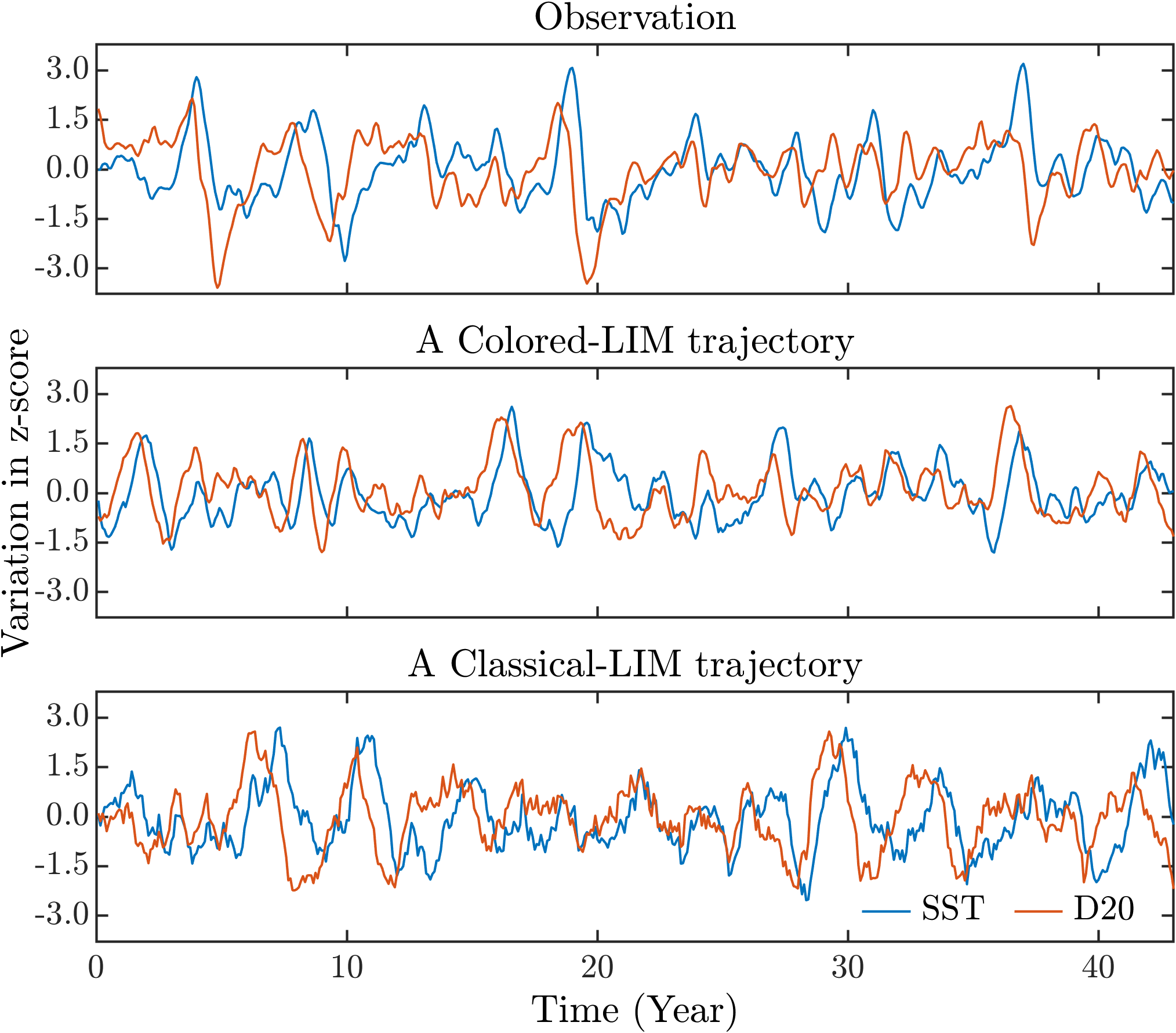}
    \caption{The observation of SST and D20 from 1979 to 2022, and a trajectory of Colored-LIM and Classical-LIM processes of length $43$ years.}
    \label{Fig:SamplePath}
\end{figure}

\begin{figure}[h]
    \centering
    \includegraphics[]{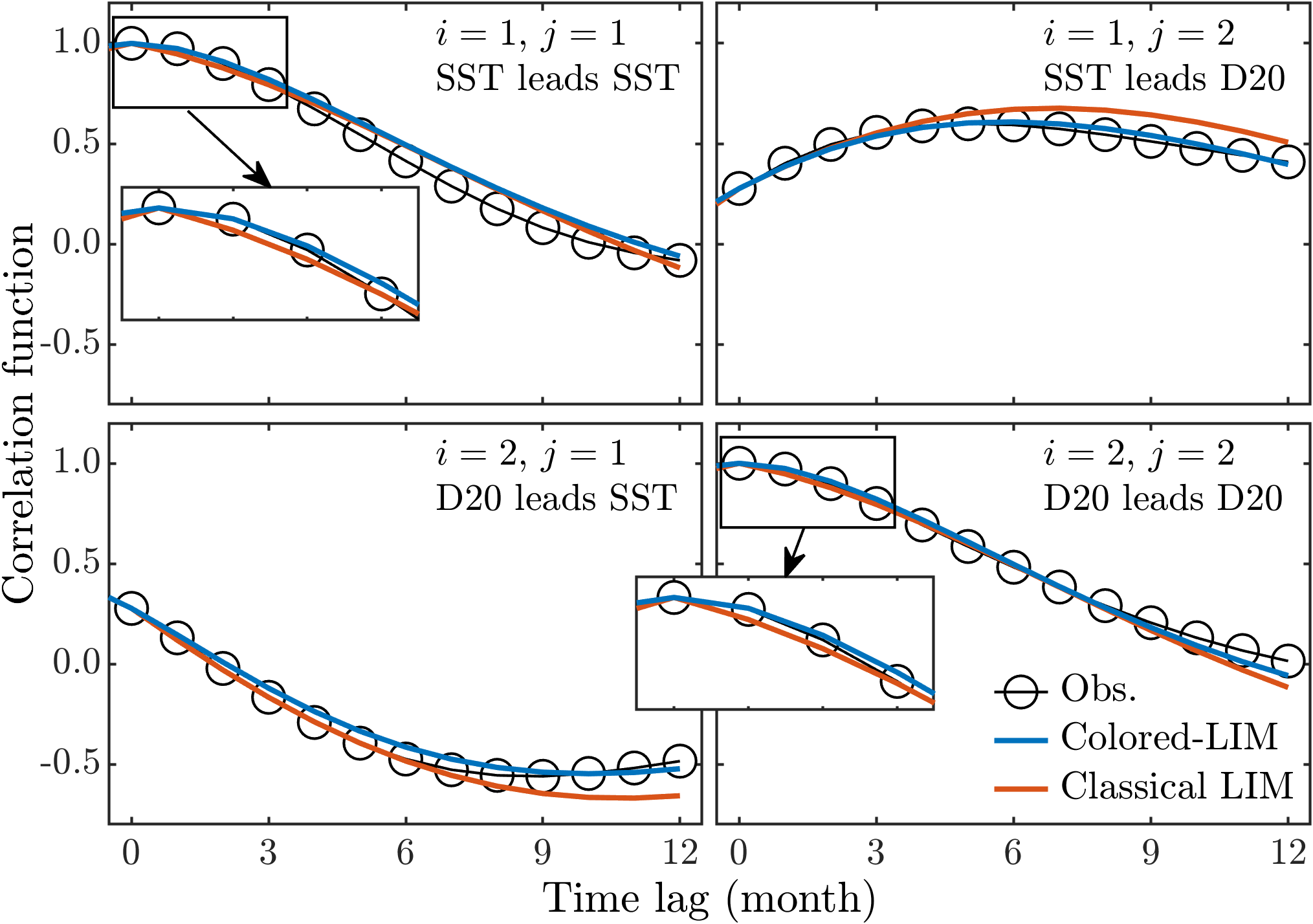}
    \caption{The $(i,j)$-entries of the observed correlation function of the normalized SST and D20 anomalies symbolized as $1$, and $2$, respectively, along with the corresponding Colored-LIM and Classical-LIM correlation functions. For clarity, the insets provide a zoomed-in view near the origin.}
    \label{Fig:CorrFunc}
\end{figure}

To quantify the performance of Colored-LIM and Classical LIM, we evaluate the ensemble prediction skill by the ensemble calibration ratio (ECR), described by
\begin{align} \label{Eq:ECR}
    \text{ECR} = \frac{\text{Forecast error}}{\text{Ensemble spread}}.
\end{align} 
ECR is a statistical measure that compares the spread of an ensemble forecast to the observed variability in the actual outcomes \cite{Perkins2020}. 
If the ECR = 1, the model is well-calibrated, meaning that the ensemble spread accurately reflects the observed variability.
If the ECR $> 1$ ($<1$), then the model is considered under-dispersive (over-dispersive), meaning that the ensemble spread is too narrow (wide) and does not fully account for the observed variability.
For a detailed implementation of ECR computation, see appendix~\ref{Chap:ECR}.

Fig.~\ref{Fig:ECR} shows that for short-term ensemble forecasting, the Colored-LIM achieves significantly better performance for both SST and D20 anomaly while the Classical LIM spreads much faster than observation, causing a conservative and over-dispersive ensemble prediction and a lower value in ECR. 
On the other hand, as the lead time goes on, the ECR for the Colored-LIM gently oscillates but stably remains close to $1$, while for the Classical LIM, the ECR is not as accurate or stable as the case of Colored-LIM.
Therefore, taking the memory of noise into account improves short-term ensemble forecasting, and the colored noise may serve as a better representation of the environmental noise.

\begin{figure}[h]
    \centering
    \includegraphics[]{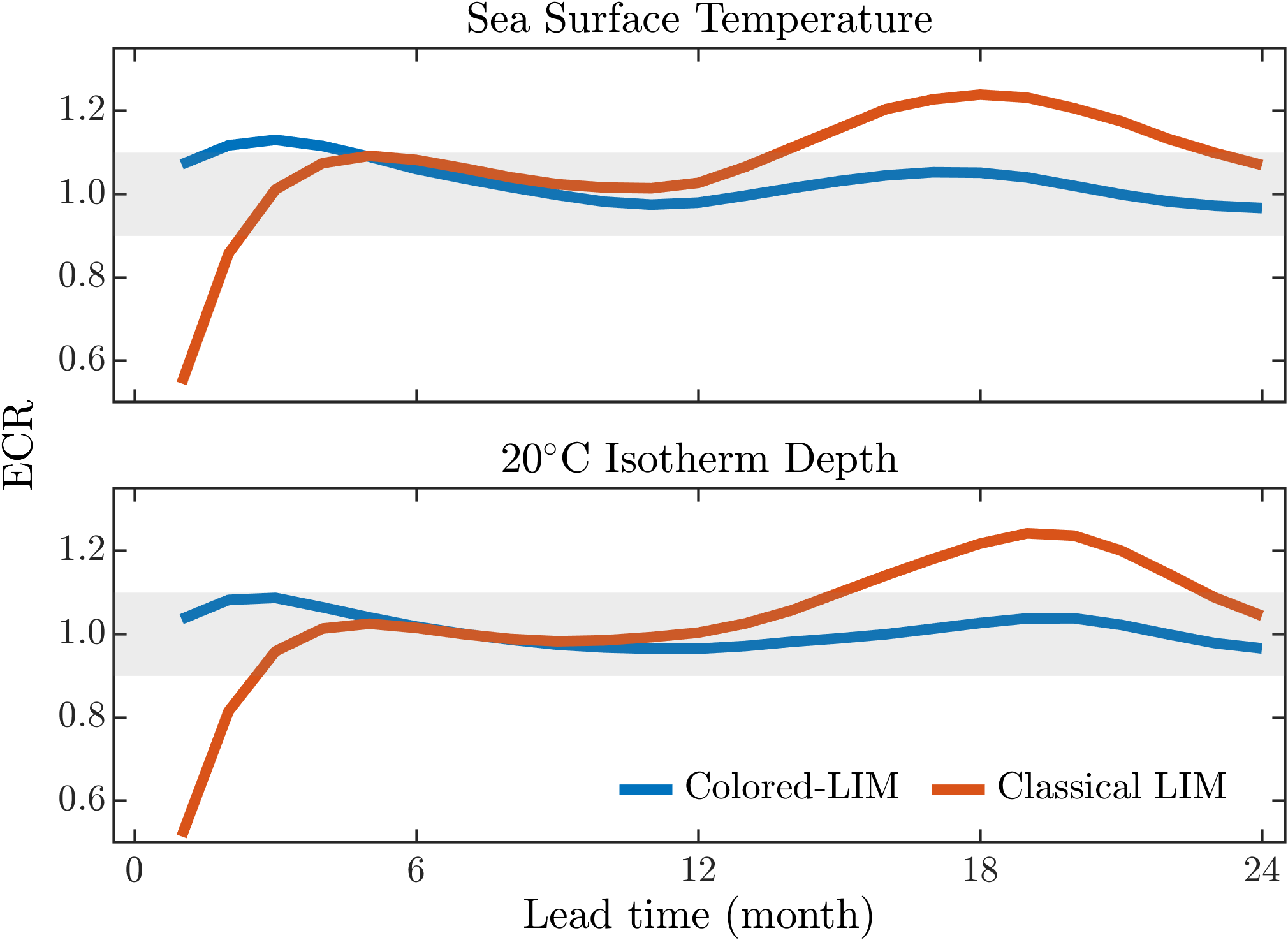}
    \caption{The ECR of SST and D20 for Colored-LIM and Classical LIM. The shaded area specifies the interval $(0.9,1.1)$, indicating a fair ensemble forecasting skill.}
    \label{Fig:ECR}
\end{figure}

\subsection{Electricity network: dynamic mode discovery}
\label{Chap:DMD}

Understanding the electricity network and the consumption pattern over a region is crucial for optimizing energy efficiency, and enhancing the grid reliability \cite{Kong2020,MOHAN2018,Sunny2020}. 
For instance, a thorough analysis helps uncover potential weaknesses within the network, allowing for proactive measures to maintain stability and prevent disruption. 
DMD offers a data-driven approach for studying such complex systems, decomposing the behavior of the electricity network into interpretable modes, and extracting coherent spatiotemporal patterns from the dataset \cite{Kutz2016,Schmid2022,Tu2014}.
As an extension of DMD, we employ the Colored-LIM to identify the key modes that capture the dominant features of the electricity pattern at Tohoku University.

We collect the time-series data on the hourly electricity consumption of $15$ clusters of buildings at Tohoku University from April 4 to July 15 (71 days), remove the hourly variation, and normalize the residuals by using z-score. 
To reduce the complexity of the data and avoid instability due to high dimensionality, we project time-series data of the normalized residuals into its $r$-dimensional PC space, where $r = 6$ is determined by the energy method such that $90\%$ of the data variance is captured.
Then, by assuming the rapidly oscillatory behavior as the white noise inference (e.g., measurement error, measurement resolution, etc.), we apply a low pass filter with a $1$-day bandpass frequency and then the Colored-LIM with a pre-selected noise correlation time $\tau = 2$ hours. 
The eigenvalues of the dynamical matrix $\Acglim$ are listed in Table~\ref{Table:TU_eig}. 
To gain insight into the daily cycle of power consumption, we select $\#2$ and $\#3$ modes whose oscillation frequency is $22.7$ hours and $e$-folding time is $10.1$ hours, corresponding to $1$-day human activity cycle and regular working hours (break time included), respectively. 
Then the dominant dynamic mode is computed by projecting the $\#2$, $\#3$ eigenvectors back to the $15$-dimensional data space.
Since these projected eigenvectors are related by complex conjugation, we denote them collectively by $\phi_\text{C-LIM}$.

\begin{table}[h]
\caption{\label{Table:TU_eig}%
The distribution of eigenvalues for the dynamical matrices $\Acglim$, $\Admd$, and $A_\text{DMD}^\prime$. 
}
\begin{ruledtabular}
\begin{tabular}{cccc}
 & \multicolumn{3}{c}{Eigenvalue (Unit: $\text{day}^{-1}$)}\\ \cmidrule(lr){2-4}
\textrm{Mode} & $\Acglim$ & $\Admd$ & $A_\text{DMD}^\prime$ \\
\midrule
$\# 1$         & $-2.7736$ & $-1.1704$ & $-54.3219$ \\
$\# 2$ & \multirow{2}{*}{ $-2.3878 \pm 1.0555 \, i$ } & \multirow{2}{*}{ $-1.0867 \pm 1.5806 \, i$ } & $-41.0924$ \\
$\# 3$ &  &  & $-15.5341$ \\
$\# 4$ & \multirow{2}{*}{ $-2.1777 \pm 0.0605 \, i$ } & \multirow{2}{*}{ $-1.0063 \pm 0.2563 \, i$ } & $-11.1664$ \\
$\# 5$ & & & $-5.5064$ \\
$\# 6$   & $-0.6645$ & $-0.4610$ & $-0.6807$ \\
\end{tabular}
\end{ruledtabular}
\end{table}

The dynamic mode has $2$ important factors: the magnitude and the angle \cite{Fang2023,Rana2023}.
The magnitude describes the derivation from the mean diurnal profile for each entry of $\phi_\text{C-LIM}$ (i.e., each cluster of buildings). 
A larger magnitude indicates that in this mode, the residuals and hence electricity consumption oscillate more intensely, so a potential consumption peak is more likely to occur in the daytime.
Fig.~\ref{Fig:TU_Analysis}a shows the magnitude of each entry, implying that the electricity consumption is less stable or predictable for building clusters No. 1, and 15. 
These clusters consist of a particularly large number of buildings, covering the primary activity areas for underclassmen and university staff, lecture rooms for advanced classes, and laboratories.
No. 2 which contains student learning areas, shows a negligible magnitude which may relate to their consistent opening hours, from 10 a.m. to 10 p.m.
No. 7, 8, 9, and 12 exhibit smaller magnitudes, which may imply a constant working pattern over time.
In fact, clusters No. 7 - 9 mainly consist of laboratories of advanced experiments, contrary to No. 1 and 15. 
Cluster No. 12 consists of only a single building, which is one of the newest buildings on the campus.
This building may use a more advanced technique to maintain a regular and stable energy consumption profile.

\begin{figure}[h]
\includegraphics{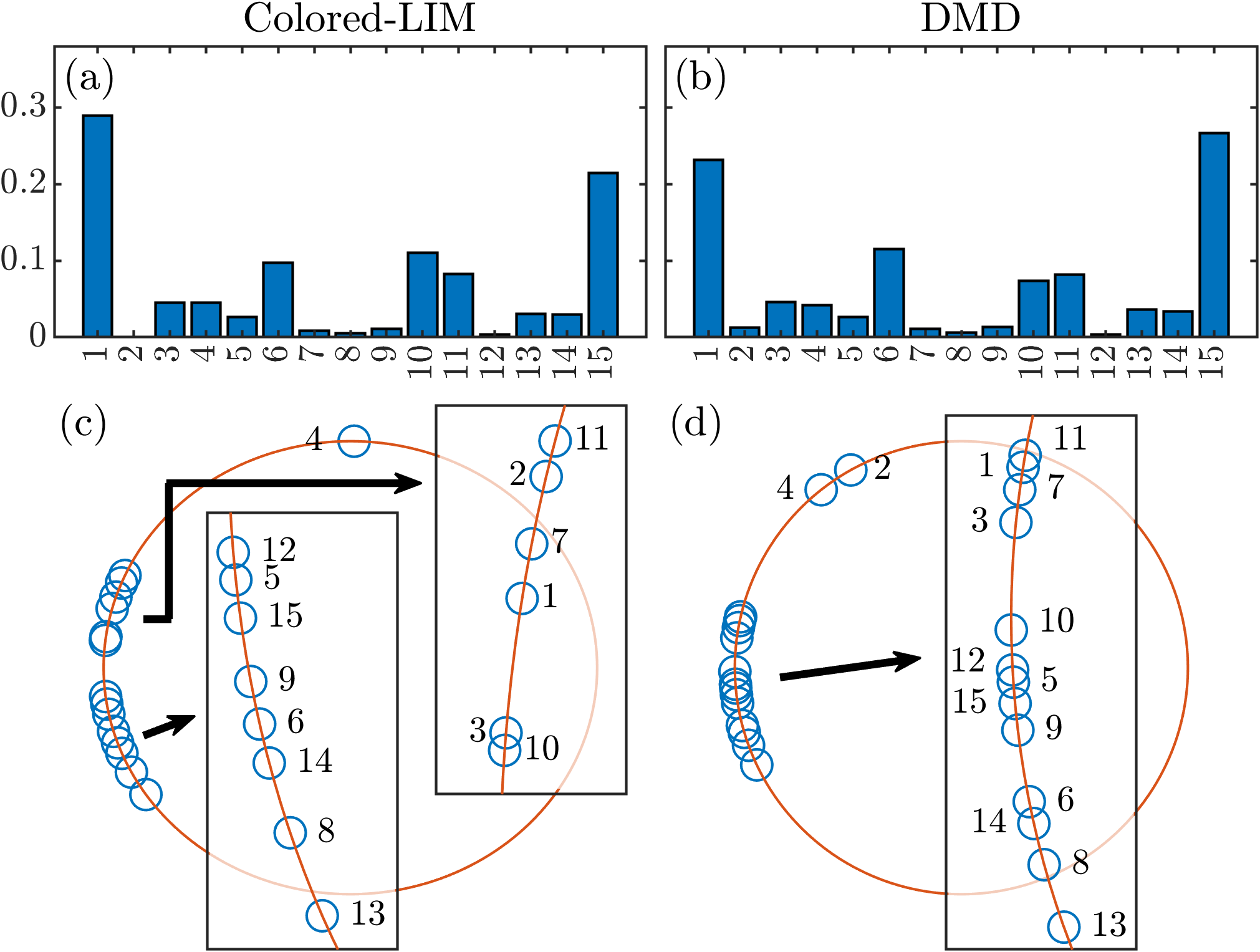}
\caption{\label{Fig:TU_Analysis} 
Magnitude and phase analysis. The upper panels, (a) and (b), show the squared absolute values for each entry in $\phi_\text{C-LIM}$ and $\phi_\text{DMD}$, respectively. The lower panels, (c) and (d), display the angle of each entry in $\phi_\text{C-LIM}$ and $\phi_\text{DMD}$ on the unit circle in the complex plane. The insets in (c) and (d) provide zoomed-in views for clarity.}
\end{figure}

On the other hand, the angle implies the phase difference between other clusters of buildings in this daily-cycle dynamic mode. 
Fig.~\ref{Fig:TU_Analysis}c shows that most clusters are located on the left-hand side of the unit circle, except cluster No. 4, which consists of buildings for student extracurricular activities and workout space.
This is probably because the electricity usage for most buildings aligns with regular working hours, but there is a distinct pattern during extracurricular activities.

In the implementation of the Colored-LIM, the noise correlation time $\tau$ is predetermined. 
In this study, the daily-cycle (i.e., $1$-day frequency) dynamic mode appears only when $\tau = 2$ hours, corresponding to the standard lecture time is $90$ minutes (no break time). 
This time scale can be understood in various ways.
For example, the cancellation of classes (as random forcing) affects student mobility and electricity usage for the next 2 hours.

To compare the Colored-LIM with the projected DMD (the Classical LIM), we follow the same workflow to obtain the eigenvalues of $\Admd$ and $A_\text{DMD}^\prime$ where the time lag is set to be $\rho = 2$ hours, with and without applying a low pass filter to the normalized residual data, respectively.
The eigenvalues are listed in Table~\ref{Table:TU_eig}.
The eigenvalues of $A_\text{DMD}^\prime$ all show purely damped behaviors, with most exhibiting highly damped behavior, which may not offer much physically meaningful interpretation.
Meanwhile, for $\Admd$, the time scale of $\#2$ and $\#3$ modes are close to that of human activities, and the corresponding dynamic mode $\phi_\text{DMD}$ exhibits similar behavior as $\phi_\text{C-LIM}$, as shown in Fig.~\ref{Fig:TU_Analysis}b.
In the phase analysis shown in Fig.~\ref{Fig:TU_Analysis}d, cluster No. 2 is also away from the other clusters, probably indicating a different consumption pattern since the student learning areas are open until late evening.

Similar implications of the Colored-LIM and DMD results do not mean that the two techniques are equivalently effective. 
Firstly, for modes $\# 2$ and $\# 3$, the damping rate and frequency of Colored-LIM result are related to human activity, providing an insightful interpretation of the electricity network. 
On the contrary, the deviation of the DMD eigenvalues from the diurnal activity pattern may imply the limitation of the white noise idealization in the traditional stochastic model. 
In addition, a further investigation shows that $\Acglim$ differs from $\Admd$ with a relative difference $\sim 35 \%$ and the modes $\# 1$, $\# 4$, and $\# 5$, though not the dominant modes in this study, exhibit a significant difference.
Therefore, it may be inferred that both Colored-LIM and DMD capture the dominant mode of the electricity consumption profile with Colored-LIM providing a more detailed dynamical characterization.

Overall, the Colored-LIM serves as an alternative data-driven technique to investigate the dynamic relevant information of a complex system.
In this case study, the use of colored noise to represent random forcing can be justified by the limited rate of information transmission, for example, through mobility, which is constrained by the time scale of events like lectures or the spatial extent of the campus area.
The dynamic mode $\phi_\text{C-LIM}$ corresponding to the daily cycle may be used to diagnose the electricity consumption pattern, deepen our understanding of the electricity network, and provide information for the energy-saving strategy in the future.


\section{Concluding Remarks}

In this article, we have proposed a novel linear inverse model called Colored-LIM to estimate the linear dynamics and the diffusive behavior from a finite realization of a complex non-linear dynamical system with temporally correlated stochasticity.
We have shown that for a given predetermined noise correction time, the linear dynamics can be estimated by solving the linear system that involves the higher-order derivatives of the correlation function of the observable. 
Then, the diffusion matrix is solved using the generalized fluctuation-dissipation relation.
As for the determination of noise correlation time, it has been formulated as a $1$-dimensional minimization problem, making the Colored-LIM numerically efficient.
In addition, we have argued that the Colored-LIM does not converge to the Classical LIM in a white noise limit, indicating that the temporal structure of environmental noise may have a significant influence on the dynamical system. 
Furthermore, the Colored-LIM has been connected with DMD by stating that it allows dependent residual terms from a regression perspective. 

We have tested Colored-LIM with linear and non-linear problems. 
It has been shown that even though a longer observation window is required whenever the noise memory is strong, the Colored-LIM returns a robust result so long as the observation data accurately reflects the statistics of the underlying process. 
Also, Colored-LIM can be thought of as learning the dynamics and stochastics from observed correlation function by a curve fitting technique.
For real-world applications, we have applied the Colored-LIM to ensemble prediction for ENSO and dynamic mode discovery for the electricity network. 
Modeling the environmental stochasticity by colored noise improves the performance and interpretability compared with the white-noise-based modeling, and the estimated stochasticity provides further insight into the complex systems.
Indeed, as an extension of Classical LIM and DMD, Colored-LIM possesses the potential to achieve a variety of tasks, making it a versatile tool.


We note that the prior knowledge can be integrated into the Colored-LIM algorithm. 
It has been shown that the noise correlation time can be predetermined in the numerical study, but indeed, in the cases where the dynamical matrix is known a prior, the Colored-LIM can be used to estimate the noise correlation time and diffusion by replacing Line 4 in Alg.~\ref{Alg:LIM-CG} with the given one.
Though it may be rarely the case, we believe that the Colored-LIM can be applied for seismology and engineering, in which the measurement equipment is a known damped oscillator \cite{Cuvalci1996,Havskov2016,Hou2019,Setareh2006,Wang2023}.

Furthermore, the Colored-LIM algorithm can simultaneously utilize various numerical differentiation methods.
In this study, implementing Alg.~\ref{Alg:LIM-CG}  with both the finite difference scheme and polynomial fitting yielded stable results, probably due to the detrending process.
However, in practice, higher-order numerical derivatives of the observed correlation function may become unstable due to the noisy component, and results can greatly vary depending on the numerical differentiation method used \cite{Chapra2018,Chartrand2011,STICKEL2010}. 
To this end, the optimization can be expanded to simultaneously minimize over both $\tau$ and all chosen numerical differentiation schemes.
Consequently, the Colored-LIM produces an accurate linear representation for the underlying complex system.

\begin{acknowledgments}
This work is supported by multiple funding sources. Justin Lien and Hiroyasu Ando are supported by the Council for Science, Technology and Innovation (CSTI), Cross-ministerial Strategic Innovation Promotion Program (SIP), the 3rd period of SIP "Smart Energy Management System" (Grant Number JPJ012207, funding agency: JST). Shoichiro Kido is supported by JSPS-Kakenhi (Grant Number: 21K13997, 23K25946, and 24H00280).
\end{acknowledgments}

\appendix

\section{Ensemble Calibration Ratio (ECR)} \label{Chap:ECR}

Let $\{ x(t) \}$ be the observed time-series data and $s$ be the forecasting lead time. 
For each time $t$, stochastic integration of LIMs with the initial condition $x(t)$ generates a set of predicted values $\{ y_m(t+s) \}$ at time $t+s$ where $m$ indexes the ensemble members.

To quantify the forecasting error at time $t+s$, the ensemble average squared error (SE) is defined by
\[ \text{SE}(t+s) = \big( x(t+s) - y(t+s) \big)^2 \]
where $y(t+s) = \langle y_m (t+s) \rangle$, with the bracket denoting the ensemble mean.
The SE measures the difference between the observed value $x(t+s)$ and the average prediction by the ensemble, thus quantifying the accuracy of the forecast.

The ensemble spread at time $t+s$ is evaluated using the ensemble variance $\sigma^2$, defined by 
\[ \sigma^2(t+s) = \langle \big( y_m(t+s) - y(t+s) \big)^2 \, \rangle. \]
The variance represents the uncertainty in the forecast, as it captures the variability among the ensemble members.
Finally, the ECR for a given lead time $s$ is computed as
\[ \text{ECR(s)} = \overline{ \left( \frac{\text{SE}(t+s)}{\sigma^2(t+s)} \right) }. \]

The ECR provides a measure of the calibration of the ensemble forecasting, comparing the forecast error (SE) to the forecast uncertainty ($\sigma^2$).
A well-calibrated ensemble will have an ECR value close to $1$, indicating that the spread of the ensemble is consistent with the observed forecast error.



\bibliography{apssamp}

\end{document}